\documentclass{amsart}
\usepackage{amsmath, amsthm, amssymb,enumerate,hyperref}

\newcommand{\cat}[1]{\ensuremath{\mathbf{#1}}}
\DeclareMathOperator{\AWstar}{\cat{AWstar}}
\DeclareMathOperator{\Cstar}{\cat{Cstar}}
\DeclareMathOperator{\Proj}{Proj}
\newcommand{\eqcomment}[1]{& \hbox{\scriptsize{(#1)}}}
\DeclareMathOperator{\card}{card}
\DeclareMathOperator{\supp}{supp}
\DeclareMathOperator{\Spec}{Spec}

\renewcommand{\d}{\ensuremath{\mathrm{d}}}
\newcommand{\dbar}{\ensuremath{\overline{\d}}}
\newcommand{\D}[1][]{\ensuremath{\mathrm{D}_{#1}}}
\newcommand{\supP}{\sup\nolimits^+}
\newcommand{\typeone}{\ensuremath{\text{I}}}
\newcommand{\typetwo}{\ensuremath{\text{II}}}
\newcommand{\typethree}{\ensuremath{\text{III}}}
\newcommand{\alg}[1][A]{\ensuremath{\mathnormal{#1}}}
\newcommand{\masa}[1][C]{\ensuremath{\mathnormal{#1}}}
\newcommand{\cover}[2][]{\ensuremath{\mathrm{c}_{#1}(#2)}}
\newcommand{\M}{\mathbb{M}}

\theoremstyle{plain}
\newtheorem{theorem}{Theorem}[section]
\newtheorem{corollary}[theorem]{Corollary}
\newtheorem{lemma}[theorem]{Lemma}
\newtheorem{proposition}[theorem]{Proposition}
\newtheorem*{comparisontheorem}{Comparison theorem}
\theoremstyle{definition}
\newtheorem{definition}[theorem]{Definition}
\newtheorem{remark}[theorem]{Remark}

\begin{document}

\title{Diagonalizing matrices over AW*-algebras}
\author{Chris Heunen}
\address{Department of Computer Science, University of Oxford, Wolfson
  Building, Parks Road, OX1 3QD, Oxford, UK}
\email{heunen@cs.ox.ac.uk}
\thanks{Support by US Office of Naval
Research grant N000141010357 is gratefully acknowledged.}
\author{Manuel L. Reyes}
\address{Department of Mathematics\\ Bowdoin College\\
  8600 College Station\\ Brunswick, ME 04011\\ USA}
\email{reyes@bowdoin.edu}
\date{\today}

\begin{abstract}
  Every commuting set of normal matrices with entries in an AW*-algebra can
  be simultaneously diagonalized.
  To establish this, a dimension theory for properly infinite
  projections in AW*-algebras is developed. 
  As a consequence, passing to matrix rings is a functor on the
  category of AW*-algebras.  
\end{abstract}

\maketitle

\section{Introduction}
\label{sec:introduction}

Diagonalization is a fundamental operation on matrices that can simplify reasoning
about normal matrices. Every commuting set of normal $n \times n$ complex matrices
can be simultaneously diagonalized.
If $\alg$ is a unital C*-algebra, it is well known that the ring $\M_n(\alg)$ of
$n \times n$ matrices with entries in $\alg$ is again a unital C*-algebra. 
The question naturally arises: over which C*-algebras can any
commuting set of normal $n \times n$ matrices be diagonalized?
To be precise, we say that $\alg$ is \emph{simultaneously $n$-diagonalizable} 
if, for any commuting set $\alg[X]$ of normal elements of $\M_n(\alg)$,
there is a unitary $u$ in $\M_n(\alg)$ making $uxu^*$ diagonal for any
$x \in \alg[X]$. 
(Note that this property is stronger than the ability to
  diagonalize individual normal $n \times n$ matrices.)
We prove that every AW*-algebra is simultaneously
$n$-diagonalizable for any positive integer $n$.

This question has quite some history.
Deckard and Pearcy first established in~\cite{DeckardPearcy} that every
individual normal matrix is diagonalizable in $\M_n(\alg)$ for a commutative
AW*-algebra $\alg$ in 1962.
In~1977, Halpern showed that a single normal element of a properly infinite von Neumann algebra
  is diagonalizable (though this seems not to have been widely
  noticed~\cite[Lemma~3.2]{halpern:essential}). Since then,
the problem of diagonalizing an individual matrix or operator has been studied
in several contexts; for a brief survey and further references
see~\cite[Chapter~6]{manuilovtroitsky:hilbert}.
The question of whether an individual self-adjoint matrix over an
AW*-algebras is diagonalizable was raised in~\cite{frankmanuilov:diagonalizing}.
Simultaneous diagonalization of matrices over noncommutative operator algebras was initiated
by Kadison in 1982 (\cite{kadison:diagonalization}, see also
~\cite[Volume~IV, Exercises~6.9.18--6.9.35]{kadisonringrose}). 
He proved that countably decomposable von Neumann algebras are
simultaneously $n$-diagonalizable, relying on their decomposition into types (see also~\cite{kaftal:types}). 
In 1984, Grove and Pedersen showed that for any $n \geq 2$, a commutative simultaneously
$n$-diagonalizable C*-algebra is an AW*-algebra, and they asked whether Kadison's techniques
extend to noncommutative AW*-algebras~\cite[6.7]{grovepedersen:diagonalizing}.
We precisely accomplish this task. 

The bulk of the new results here concerns properly infinite AW*-algebras.
In that case, our attack on the question requires a dimension theory,
reducing equivalence of properly infinite projections to a problem
about cardinal-valued dimensions.
Kadison sidestepped such size issues by restricting to countably decomposable von
Neumann algebras. By proving everything in full generality, our results
are even new in the case of properly infinite von Neumann algebras.
A dimension theory for AW*-algebras was given by Feldman already in
1956~\cite{feldman:dimension}. Independently, \v{C}ilin studied a similar
notion of dimension in 1980.\footnote{Apparently it was published
in~\cite{chilin:equivalence}, but we did not manage to locate
that paper; instead we re-engineered, and generalized, the proofs of
Theorems~\ref{thm:cilin1} and~\ref{thm:cilin2} below from Proposition
3.6.6 in \v{C}ilin's thesis. We thank S.\ Solovjovs for obtaining that
thesis, and A.\ Akhvlediani for translating that proposition.} 
Tomiyama greatly extended Feldman's results in the case of von Neumann algebras
in 1958~\cite{tomiyama:dimension}; for a recent, and very general, account, see~\cite{goodearlwehrung}. 
However, these studies into dimension theory do not interface 
seamlessly with Kadison's diagonalization results. Therefore, either
the dimension theory or Kadison's methods have to be adapted; we chose
the former. This requires a more intricate analysis of the dimension
function. A crucial step here is a decomposition into so-called
equidimensional projections. As a side note, we must mention that all
these results depend heavily on the axiom of choice, and therefore are
problematic in constructive settings. 

Our original interest in diagonalization over AW*-algebras arose from the
following problem.
Let $\Cstar$ denote the category whose objects are unital C*-algebras
and whose morphisms are unital $*$-homomorphisms. Let
$\AWstar$ denote the subcategory of $\Cstar$ whose objects
are the AW*-algebras and whose morphisms are those $*$-homomorphisms
that preserve suprema of arbitrary sets of projections.
Applying $*$-homomorphisms entrywise makes $\M_n$ into a functor $\Cstar
\to \Cstar$. On
objects, this functor sends AW*-algebras to AW*-algebras, by a
combination of results due to Kaplansky and Berberian~\cite{berberian:matrices}.
So it is natural to ask whether $\M_n$ restricts to a functor
$\AWstar \to \AWstar$. As an application of the
diagonalization theorem, we prove that this is indeed so.

As is clear from the historical introduction above, there is a fair
amount of (routine) generalization from von Neumann algebras to
AW*-algebras,\footnote{See also Remark~\ref{rem:baer}.} as well as piecing 
together fragmented results from the literature. To make the story
reasonably self-contained, we include all such results in a uniform
way with explicit proofs, relying upon~\cite{berberian} as our standard reference
for the theory of AW*-algebras. The paper is structured as follows. 
After discussing preliminaries in Section~\ref{sec:preliminaries}, and the
routine generalizations of Kadison's results to AW*-algebras of
finite type in Section~\ref{sec:comparison:finite}, 
the next few sections launch into the proof of simultaneous $n$-diagonalizability
of AW*-algebras. 
Section~\ref{sec:dimension} introduces
the dimension theory, which is continued in
Section~\ref{sec:equidimensional}, that concerns equidimensional
projections. The dimension theory is then put to use in
Section~\ref{sec:comparison:infinite} to generalize Kadison's results to
AW*-algebras of infinite type. Then Section~\ref{sec:diagonalization} 
gathers all the ingredients to prove that AW*-algebras are
simultaneously $n$-diagonalizable. 
Section~\ref{sec:functor} ends the paper with the
functoriality of taking matrix rings of AW*-algebras. Finally,
Appendix~\ref{app:supachieved} contains additional technical results about
dimensions that would disrupt the main development.
Some open questions are mentioned at the end of Sections~\ref{sec:comparison:infinite}
and~\ref{sec:diagonalization}.

\section{Preliminaries on AW*-algebras}\label{sec:preliminaries}

An \emph{AW*-algebra} is a C*-algebra in which the (right, and
hence left) annihilator of any subset is generated by a single projection.
This section recalls some general properties of these algebras, which were
introduced by Kaplansky as a generalization of von Neumann algebras, preserving
the purely algebraic content of their theory~\cite{kaplansky:awstar}.
For example, the Gelfand spectrum of a commutative AW*-algebra is a
Stonean space (\textit{i.e.}~a topological space in which the closure
of an open set is again open, \textit{i.e.}~the Stone space of a complete
Boolean algebra). 
To compare: the Gelfand spectrum of a commutative von Neumann algebra
additionally satisfies a measure-theoretic property.

\subsection*{Maximal abelian subalgebras}
 
We will use the abbreviated phrase \emph{maximal abelian subalgebra} in place
of ``maximal abelian $*$-subalgebra'' or ``maximal abelian self-adjoint
subalgebra''. The notion of AW*-subalgebra is slightly subtle, but
maximal abelian subalgebras are automatically AW*-subalgebras.

\subsection*{Projections}

The main characteristic of AW*-algebras is that  to a great extent
they are algebraically determined by their projections.
For example, any AW*-algebra $\alg$ is the closed linear span of its
projections $\Proj(\alg)$.
Projections are partially ordered by $e \leq f$ if and only if $e = ef (=
fe)$, and $\Proj(A)$ is a complete 
lattice. In the special case that $\{e_i\}$ is an orthogonal set of
projections in $A$, we denote its supremum by $\sum e_i$.
Projections $e,f \in \Proj(\alg)$ are equivalent when
$e=vv^*$ and $f=v^*v$ for some $v \in \alg$. When the algebra in which
they are equivalent must be emphasized, we write $e \sim_{\alg} f$,
and similarly for the derived notions $e \precsim_{\alg} f$
(meaning $e \sim e' \leq f$ for some projection $e'$)
and $e \prec_{\alg} f$ (meaning $e
\precsim f$ but $e \not\sim f$; we also allow $0 \prec 0$). 
Equivalence is additive: if $\{e_i\}$ and $\{f_i\}$ are orthogonal families of
projections satisfying $e_i \sim f_i$, then $\sum e_i \sim \sum f_i$.
Equivalence also satisfies Schr{\"o}der--Bernstein: if $e \precsim f$
and $e \succsim f$, then $e \sim f$.
It is a simple fact that if $z, e, f \in \Proj(\alg )$ are such that
$z$ is central and $e \sim f$, then $ze \sim
zf$. 

\begin{comparisontheorem}
  Let $e$ and $f$ be projections in an AW*-algebra. There are 
  orthogonal central projections $x,y,z$ satisfying $x+y+z=1$ and
  \[
    xe \prec xf, \qquad ye \sim yf, \qquad ze \succ zf.
  \]
\end{comparisontheorem}
\begin{proof}
  Zorn's lemma produces a maximal orthogonal family $\{y_i\}$ of
  nonzero central projections satisfying $y_ie \sim y_if$. Setting $y=\sum
  y_i$, then $ye \sim yf$. In fact, this $y$ is the unique largest
  central projection with that property: if $we \sim wf$ for $w \in
  \Proj(Z(\alg))$ then $(1-y)we \sim (1-y)wf$, but $(1-y)w$ is
  orthogonal to $y$ and must hence be zero by maximality of $\{y_i\}$.

  There is a central projection $w$ such that $we \precsim wf$
  and $(1-w)e \succsim (1-w)f$~\cite[Corollary~14.1]{berberian}. Set
  $x=w(1-y)$ and $z=(1-w)(1-y)$. Then $x$, $y$ and $z$ are orthogonal
  and sum to 1. Clearly also $xe \precsim xf$, and because $xe \sim xf$
  violates maximality of $y$ as above, in fact $xe \prec xf$. Similarly
  $ze \succ zf$.
\end{proof}

\noindent
In fact, the $x$, $y$ and $z$ in the comparison theorem are unique, but
we do not need this fact.

\subsection*{Passing to corner algebras}

We will frequently use properties of corners of an AW*-algebra $\alg$, which we
now list. For any $e \in \Proj(\alg)$, the corner algebra $e\alg e$ and the centre
$Z(\alg)$ are again AW*-algebras. 
Many relevant properties are preserved by passing to corners. 
For example, the following lemma shows that equivalence and maximality of abelian
subalgebras are also well behaved when passing to corners.

\begin{lemma}\label{lem:corners}
  Let $e$ be a projection in an AW*-algebra $\alg$.
  \begin{enumerate}[(a)]
  \item $\Proj(e\alg e) = \{ p \in \Proj(\alg ) \mid p \leq e \}$;
  \item For all $p,q \in \Proj(e\alg e)$ we have $p \sim_{\alg}  q$ if and only if $p \sim_{e\alg e} q$.
  \item   If $\masa$ is a maximal abelian subalgebra of $\alg$, and $e
  \in \Proj(\masa)$, then $e\masa$ is a maximal abelian subalgebra of $e\alg e$.
  \end{enumerate}
\end{lemma}
\begin{proof}
  By definition $p \in e\alg e$ if and only if $p=eae$ for some $a \in
  \alg$. This is equivalent to $p=ep=pe$, that is, to $p \leq e$,
  establishing (a).
  For the non-trivial direction of (b), suppose $p
  \sim_{\alg}  q$, say $v^*v=p$ and $vv^*=q$. Since we may assume that
  $v \in \alg$ is a partial isometry~\cite[Proposition~1.6]{berberian},
  $v = vv^*vv^*v = qvp \in q\alg p \subseteq e\alg e$,
  so $p \sim_{e\alg e} q$.

  For (c), observe that for any projection $z$ in $\masa$ one has
  $z(1-e) \leq 1-e$ in $\alg$, and so 
  $ez(1-e)=0$. Similarly $(1-e)ze=0$.  If $eae \in e\alg e$ commutes
  with $e\masa$, then 
  \[  
    eaez = eaeze+eaez(1-e) = eaeze = ezeae + (1-e)zeae = zeae.
  \]
  Hence $eae$ commutes with $\masa$. So $eae \in \masa$ by maximality of $\masa$.
  Therefore $eae \in e\masa$, and $e\masa$ is maximal.
\end{proof}

\subsection*{Central covers}

We write $\cover{e}$ for the least central projection above $e$, also
called its central cover. If the AW*-algebra $\alg$ must be emphasized, we write
$\cover[\alg]{e}$ instead. 
%
%
%
Central covers and centres are also preserved by passing to corners.

\begin{lemma}\label{lem:centralcoverincorners}
   If $f \leq e$ are projections in an AW*-algebra $\alg$, then
   $\cover[e\alg e]{f} = \cover[\alg ]{f} e$. Hence $Z(e\alg e) = eZ(\alg )$.
\end{lemma}
\begin{proof}
  See Proposition~6.4 and Corollary~6.1 of~\cite{berberian}.
%
\end{proof}

We record two results of Kadison's on central covers, adapted to AW*-algebras.

\begin{lemma}\label{lem:KR6.9.18}
  If $\masa$ is a maximal abelian subalgebra of an AW*-algebra $\alg$ and $\masa 
  \neq \alg$, then there are nonzero orthogonal projections $e,f$ in $\masa$
  with $\cover{e}=\cover{f}$ and $e \precsim f$.
\end{lemma}
\begin{proof}
  If $p \in \Proj(\masa )$ satisfies $\cover{p}\cover{1-p}=0$, then $p=\cover{p}$,
  because
  \[
     p \leq \cover{p} \leq 1-\cover{1-p} \leq 1-(1-p) = p.
  \]
  So either each projection in $\masa$ is central in $\alg$, or
  $q=\cover{p}\cover{1-p} >0$ for some projection $p$ in $\masa$. The former case
  is ruled out, because then $Z(\alg)=\masa$, and hence
  $\masa =\alg$ by maximality. Now $qp$ and $q(1-p)$ are nonzero and
  $\cover{qp}=\cover{q(1-p)}$. By the comparison
  theorem, there is a nonzero central
  projection $z \leq q$ with either $zp \precsim z(1-p)$ or $z(1-p)
  \precsim zp$. In any event, one of $zp$ and $z(1-p)$ serves as $e$
  and the other as $f$, when $\alg$ is not abelian.
\end{proof}

\begin{lemma}\label{lem:KR6.9.19}
  If $\alg$ is an AW*-algebra without abelian central summands, then any
  maximal abelian subalgebra $\masa$ contains a projection $e$ with
  $\cover{e}=1=\cover{1-e}$ and $e \precsim 1-e$.
\end{lemma}
\begin{proof}
  Let $\{e_i\}$ be a family of nonzero projections in $\masa$ maximal with
  respect to the properties that $\{\cover{e_i}\}$ is orthogonal and $e_i
  \precsim 1-e_i$ for each $i$. From Lemma~\ref{lem:KR6.9.18}, $\masa$
  contains nonzero orthogonal projections $e_0 \precsim f_0$ ($\leq
  1-e_0$). Thus the family $\{e_i\}$ is not empty. Set $e=\sum_i
  e_i$. Then $\cover{e}=\sum \cover{e_i}$. If $z=\cover{e}<1$, then $(1-z) \alg$ is a
  nonabelian AW*-algebra (since $\alg$ is assumed to have no central 
  summands that are abelian) and $(1-z) \masa$ is a maximal abelian
  subalgebra. Again from Lemma~\ref{lem:KR6.9.18}, there is a nonzero
  projection $e_1$ in $(1-z) \masa$ with $e_1 \precsim
  (1-z)-e_1$. Adjoining $e_1$ to $\{e_i\}$ contradicts maximality of
  that family. Thus $z=1$. Since
  \[
    e_i = \cover{e_i} e_1 \precsim \cover{e_i} (1-e_i) = \cover{e_i}-e_i
  \]
  for each $i$, we have
  $
    e = \sum_i e_i \precsim \sum_i (\cover{e_i}-e_i) = 1-e
  $,
  and $\cover{e}=z=1$.
\end{proof}

\subsection*{Properly infinite projections}

A projection $e$ is \emph{finite} when $e \sim f \leq e$ implies
$e=f$; otherwise it is \emph{infinite}. 
It follows from the comparison theorem that if $\cover{e} \leq
\cover{f}$ for a finite projection $e$ and an infinite projection $f$, then
$ze \prec zf$ for nonzero central projection $z \leq \cover{f}$.
Following standard
terminology, an AW*-algebra $\alg$ is \emph{properly infinite} if every
nonzero central projection of $\alg$ is infinite.  

\begin{lemma}\label{lem:properlyinfinite}
  Let $\alg$ be an AW*-algebra, and let $e$ be a nonzero projection in $\alg$. The following are equivalent:
  \begin{enumerate}[(a)]
  \item there exist projections $e_1 \sim e_2 \sim e$ in $\alg$ such that
    $e = e_1 + e_2$;
  \item there exists an infinite orthogonal set of projections $\{ e_i \}$ in $\alg$
    such that $e_i \sim e$ for all $i$ and $e = \sum e_i$;
  \item the AW*-algebra $e\alg e$ is properly infinite;
  \item if $z \in \alg$ is a central projection, then $ze$ is either
    zero or infinite.
  \end{enumerate}
\end{lemma}
\begin{proof}
  That $(c) \Rightarrow (b) \Rightarrow (a)$ is essentially~\cite[Theorem~17.1]{berberian}.
  Suppose~(d) holds, and let $x$ be a nonzero central projection in
  $e\alg e$.  Lemma~\ref{lem:centralcoverincorners} provides
  $z \in \Proj(Z(\alg))$ such that $x=ze$, which is infinite in
  $\alg$ by assumption. So $x \sim_{\alg} f < x$ for some $f$ in
  $\alg$. But then $f \in e\alg e$ satisfies $x \sim_{e\alg e} f <
  x$ in $e \alg e$ by Lemma~\ref{lem:corners}. Hence $x$
  is infinite in $e\alg e$, establishing~(c).
  Finally, we prove $(a)\Rightarrow(d)$. Let $z$ be a central
  projection in $\alg$ such that $ze>0$. Then $ze_2 \sim ze$ is nonzero
  whence $ze \sim ze_1 < ze_1+ze_2=ze$. So $ze$ is infinite.
\end{proof}

A nonzero projection $e$ in an AW*-algebra is \emph{properly infinite} if it satisfies the
equivalent conditions of the previous lemma. Being properly infinite
is preserved by equivalence of projections (\textit{e.g.}~by Lemma~\ref{lem:properlyinfinite}(c)).
It also follows from the previous lemma that $ze$ is properly infinite
for any nonzero central projection $z \leq \cover{e}$.

\begin{lemma}\label{lem:infiniteincorners}
  Let $e$ be a projection in an AW*-algebra $\alg$.
  \begin{enumerate}[(a)]
  \item A projection in $e\alg e$ is properly infinite in $\alg$ if and only
    if it is so in $e\alg e$.
  \item If $e$ is infinite, there is a central projection $z \in \alg$ making
    $ze$ finite and $(1-z)e$ properly infinite.
  \end{enumerate}
\end{lemma}
\begin{proof}
  For (a), let $f \in e \alg e$ be a projection that is properly infinite in
  $\alg$. Then $f = a_1+a_2$ and $a_1 \sim_{\alg}  a_2 \sim_{\alg}  f$ for some
  $a_1,a_2 \in \Proj(e \alg e)$ by Lemma~\ref{lem:properlyinfinite}. But
  since $a_i \leq f$, in fact $a_i \in \Proj(e\alg e)$ and $a_1
  \sim_{e\alg e} a_2 \sim_{e\alg e} f$ by Lemma~\ref{lem:corners}. So 
  $f$ is properly infinite in $e\alg e$. The converse is trivial.

  For (b), let $\{z_i\}$ be a maximal orthogonal family of nonzero central
  projections such that $z_ie$ is finite for each $i$. Set $z=\sum z_i$. Then $ze$
  is finite~\cite[Proposition~15.8]{berberian}. Moreover, if $y$ is a central
  projection such that $y(1-z)e$ is finite, then $y(1-z)$ must be zero by
  maximality of $\{z_i\}$. So $(1-z)e$ is properly infinite by
  Lemma~\ref{lem:properlyinfinite}(d).
\end{proof}

\subsection*{Decomposition into types}

Another property that survives passing to corners is the decomposition
into types of an AW*-algebra. Recall that a projection $e$ is abelian
when $e \alg e$ is abelian. An AW*-algebra is of type $\typeone$
if it has an abelian projection with central cover 1; it is of type
$\typetwo$ if it has a finite projection with central cover 1 but no
nonzero abelian projections; and it is of type $\typethree$ if it has
no nonzero finite projections. More specifically, type $\typetwo_1$
means type $\typetwo$ and finite; type $\typetwo_\infty$ means
type $\typetwo$ and properly infinite; type $\typeone_\infty$ means
$\typeone$ and properly infinite. (Notice that the zero algebra is
of all types.)

\begin{lemma}\label{lem:KR6.9.16}
  Let $e$ be a nonzero projection in an AW*-algebra $\alg$.
  \begin{enumerate}[(a)]
  \item $e\alg e$ is finite if $\alg$ is finite;
  \item $e\alg e$ is of type $\typeone$ if $\alg$ is of type $\typeone$;
  \item $e\alg e$ is of type $\typeone_\infty$ if $\alg$ is of type
    $\typeone_\infty$ and $e$ is properly infinite;
  \item $e\alg e$ is of type $\typetwo_1$ if $\alg$ is of type $\typetwo_1$;
  \item $e\alg e$ is of type $\typetwo_\infty$ if $\alg$ is of type
    $\typetwo_\infty$ and $e$ is properly infinite;
  \item $e\alg e$ is of type $\typethree$ if $\alg$ is of $\typethree$.
  \end{enumerate}
\end{lemma}
\begin{proof}
  First, notice that a projection $p$ in $e\alg e$ is abelian in $e\alg e$ if and
  only if $pe\alg ep=p\alg p$ is abelian, if and only if $p$ is abelian in $\alg$.

  For (a), suppose $f \in \Proj(e \alg e)$ is
  finite in $\alg$. That means that $p \leq f \sim_{\alg}  p$
  implies $p=f$ for all $p \in \Proj(\alg)$. As $f \leq e$, Lemma~\ref{lem:corners}(b)
  makes this equivalent to: $p \leq f \sim_{e\alg e} p$ implies $p=f$
  for $p \in \Proj(e\alg e)$. But this means that $f$ is finite in $e\alg e$.  
  Part (b)
  is~\cite[Exercise~18.2]{berberian}. For (c): $e\alg e$
  is of type $\typeone$ by (a), and contains a properly infinite
  projection $e$ by Lemma~\ref{lem:infiniteincorners}. Part (d) follows
  from (a) and the above observation about abelian projections. Part
  (f) follows from Lemma~\ref{lem:infiniteincorners}. 

  Finally, we turn to (e). If $\alg$ is of type $\typetwo_\infty$, it has
  a finite projection $f$ with $\cover{f}=1$, and no nonzero abelian
  projections. So, by the above observation, also $e\alg e$ has no nonzero
  abelian projections. Because $e$ is properly infinite and $f$
    is finite, it follows from the comparison theorem that
  $\cover{e}f \prec e$. Thus $\cover{e}f \sim
  e_0 < e$ for some finite projection $e_0$ with $\cover{e_0}=\cover{e}$. It now
  follows from Lemmas~\ref{lem:corners}(c) and~\ref{lem:centralcoverincorners} 
  that $e_0$ is finite in $e\alg e$
  with $\cover[e\alg e]{e_0}=e$. Finally, $e$ is properly infinite in $e\alg e$ by
  Lemma~\ref{lem:infiniteincorners}(a), making $e\alg e$ of type $\typetwo_\infty$.
\end{proof}

\section{Relative comparison for AW*-algebras of finite type}
\label{sec:comparison:finite}

We begin by quickly disposing of the relative comparison theory
for AW*-algebras of finite type.  This involves relatively straightforward
generalizations of Kadison's results to AW*-algebras; the section is
included in the interest of completeness.
The results of this section and Section~\ref{sec:comparison:infinite} will
show that we can always find projections with various properties, not just
in an AW*-algebra $\alg$, but in any maximal abelian subalgebra $\masa$
of $\alg$. In this context, whenever we mention without specification
concepts such as $\sim$, $\d$, finite, infinite, abelian, or central cover,
we mean the corresponding concepts in $\alg$ (and not in $\masa$).
We start by considering AW*-algebras of type $\typetwo_1$.

\begin{proposition}\label{prop:KR6.9.27}
  Let $n$ be a positive integer, and let $\alg$ be an AW*-algebra of type
  $\typetwo_1$. Let $\masa$ be a maximal abelian subalgebra and $e \in \Proj(\masa )$.
  \begin{enumerate}[(a)]
  \item There is a sequence $e_0,e_1,e_2,\ldots \in \Proj(\masa )$ with
    $e_0=e$, $\cover{e_i}=\cover{e}$, $e_i \leq e_{i-1}$, and $e_i \precsim e_{i-1}-e_i$.
  \item If $f \in \Proj(\alg )$ satisfies $\cover{e}\cover{f} \neq 0$, then there is
    a nonzero $g \in \Proj(\masa )$ with $g \leq e$ and $g \precsim f$.
  \item If $f \in \Proj(\alg )$ satisfies $f \precsim e$, then $f \sim e_1
    \leq e$ for some $e_1 \in \Proj(\masa )$.
  \item $\masa$ contains $n$ orthogonal equivalent projections with sum $1$.
  \end{enumerate}
\end{proposition}
\begin{proof}
  (a) If $e=0$, choose $e_i=0$ for each $i$. Suppose $e>0$. Then $e\alg e$
  is of type $\typetwo_1$ by Lemma~\ref{lem:KR6.9.16}, and $e\masa e$ is a
  maximal abelian subalgebra. In particular, $e\alg e$ has no abelian
  central summands. Lemma~\ref{lem:KR6.9.19} gives $e_1 \in
  \Proj(e\masa e)$ with $\cover[e\alg e]{e_1}=e$ and $e_1 \precsim e-e_1$. It
  follows from Lemma~\ref{lem:centralcoverincorners}
  that $\cover[\alg ]{e_1}=\cover[\alg ]{e}$. 
  Induction now provides a sequence with the desired properties.

  For (b): replacing $\alg$, $\masa$, $e$ and $f$ by $z\alg$, $z\masa$, $ze$ and
  $zf$ for $z=\cover{e}\cover{f}$, we may assume that $\cover{e}=\cover{f}=1$.
  Now, if $ye_i \not\prec yf$ for each nonzero central projection $y$, then
  $f \precsim e_i$ by the comparison theorem. If
  $f \precsim e_i$ for each $i$, then $e_{i-1}-e_i$ has a
  subprojection equivalent to $f$ for each $i$. In this case $\alg$
  contains an infinite orthogonal family of projections equivalent to
  $f$, which contradicts the assumption that $\alg$ is finite. Thus,
  $ye_i \prec yf$ for some $i$ and some nonzero central $y$. Now
  $ye_i$ will serve as $g$.

  For (c), let $S$ be the set of pairs consisting of orthogonal families
  $\{e_i \in \masa  \mid i \in \alpha \}$ and $\{f_i \in \alg  \mid i \in \alpha \}$ of nonzero
  projections, where $e_i \sim f_i$ for all $i \in \alpha$, and $e_i
  \leq e$, and $f_i \leq f$. We can partially order $S$ by 
  \[
    (\{e_i \mid i \in \alpha\},\{f_i \mid i \in \alpha \}) \leq (\{e'_j
    \mid j \in \beta\},\{f'_j \mid j \in \beta\})
  \]
  when $\{e_i\} \subseteq \{e'_j\}$ and $\{f_i\} \subseteq \{f'_j\}$. 
  Zorn's lemma provides a maximal element
  $(\{e_i\},\{f_i\})$ in $S$. Set $e_1=\sum_i e_i$ and $f_1=\sum_i
  f_i$. Then $e_1 \sim f_1$ by additivity of equivalence. Now $e_1
  \in \masa$, $e_1 \leq e$, and $f_1 \leq f$. Because $\alg$ is finite and $f
  \precsim e$, we have $f-f_1 \precsim e-e_1$~\cite[Proposition~17.5, Exercise~17.3]{berberian}.
  From (b), there is a nonzero $e_0 \in \Proj(\masa )$ with $e_0 
  \leq e-e_1$ and $e_0 \sim f_0 \leq f-f_1$. But then
  $(\{e_0\} \cup \{e_i\},\{f_0\} \cup \{f_i\})$ is an element of $S$ properly larger
  than $(\{e_i\},\{f_i\})$, contradicting maximality. It follows that
  $e_1 \sim f_1 = f$.

  Finally, we turn to (d). By~\cite[Theorem~19.1]{berberian} there are
  $n$ orthogonal equivalent projections $f_1,\ldots,f_n$ in $\alg$ with
  sum $1$ since $\alg$ has type $\typetwo$. Part
  (c) gives $e_1$ in $\Proj(\masa )$ with $e_1 \sim
  f_1$. From~\cite[Proposition~17.5]{berberian}, $1-e_1 \sim 1-f_1$
  ($=f_2+\cdots+f_n$). Again from (c), there is $e_2 \leq 1-e_1$ in
  $\masa$ with $e_2 \sim f_2$. Continuing in this way, we find
  $e_1,\ldots,e_n \in \Proj(\masa )$ with $e_i \sim f_i$ and
  $e_1+\cdots+e_n \sim f_1+\cdots+f_n=1$. Since $\alg$ is finite,
  $e_1+\cdots+e_n=1$. 
\end{proof}

Next, we turn to AW*-algebras of type $\typeone_n$ ($n=1,2,3,\ldots$): finite algebras of
type $\typeone$ that have an orthogonal family $\{e_1,\ldots,e_n\}$ of
equivalent abelian projections that sum to 1. Equivalently, such
algebras are $*$-isomorphic to $\M_n(\alg[C])$ for a commutative
AW*-algebra $\alg[C]$.

\begin{lemma}\label{lem:KR6.9.21}
  Let $\alg$ be an AW*-algebra of type $\typeone$ with no infinite
  central summand. For each positive integer $n$, let $z_n$ be a central
  projection in $\alg$ such that $z_n\alg$ is of type $\typeone_n$.
  Let $\masa$ be a maximal abelian subalgebra of $\alg$.
  \begin{enumerate}[(a)]
  \item Some nonzero subprojection of $z_n$ in $\masa$ is abelian in $\alg$.
  \item $\masa$ contains an abelian projection with central cover $1$.
  \end{enumerate}
\end{lemma}
\begin{proof}
  Part (a) is proved by induction on $n$. If $n=1$, then $z_1$ is a
  nonzero abelian projection in $Z(\alg ) \subseteq \masa$. If $n>1$, then
  $z_n \alg$ is an AW*-algebra without abelian central summands, and
  $z_n \masa$ is a maximal abelian subalgebra. From
  Lemma~\ref{lem:KR6.9.19}, $z_n \masa$ contains a projection $e_1$ with
  $\cover{e_1}=z_n$ and $e_1 \precsim z_n - e_1$. Now $e_1\alg e_1$ is a type
  $\typeone$ AW*-algebra without infinite central summands
  by Lemma~\ref{lem:KR6.9.16}. Again, either $e_1\masa$ has a nonzero
  abelian projection $f$, in which case $f\alg f=fe_1\alg e_1f$ is abelian and
  $f$ is an abelian projection in $\alg$, or there is a nonzero
  projection $e_2$ in $e_1\masa$ with $e_2 \precsim e_1-e_2$. Continuing
  in this way, we produce either a nonzero abelian projection in $\masa$
  or a set of $n$ nonzero projection $e_1,\ldots,e_n$ in $z_n\alg$ with
  $e_{j+1} \precsim e_j-e_{j+1}$, $e_1 \precsim z_n-e_1$, and $e_{j+1}
  < e_j$. If $y=\cover{e_n}$, then 
  \[
    e_n, y(e_{n-1}-e_n), y(e_{n-2}-e_{n-1}), \ldots, y(e_1-e_2), y(z_n-e_1)
  \]
  are $n+1$ orthogonal projections in $y\alg$ with the same (nonzero)
  central cover, contradicting the fact that $y\alg = yz_n\alg$ is of type
  $\typeone_n$~\cite[Proposition~18.2(2)]{berberian}. Thus the process
  must end with a nonzero abelian subprojection of $z_n$ in $\masa$ before
  we construct $e_n$.

  For (b), let $\{e_i\}$ be a family of nonzero projections in $\masa$
  abelian for $\alg$ and maximal with respect to the property that
  $\{\cover{e_i}\}$ is orthogonal. Set $p=\sum \cover{e_i}$. If $p \neq 1$,
  then $(1-p) \alg$ is an AW*-algebra of type $\typeone$ with no infinite
  central summand. So~\cite[Theorem~18.3]{berberian} implies that there is
  a nonzero central projection $z \leq 1-p$ and a positive integer $n$ such that
  $z(1-p)\alg$ has type~$\typeone_n$.
  From part (a), the maximal abelian subalgebra $z\masa$ of $z \alg$
  contains a nonzero abelian projection $e_0$. But then we may adjoin $e_0$
  to $\{e_i\}$, contradicting maximality. Thus $p=1$. Now $\sum e_i$ is abelian
  for $\alg$~\cite[Proposition~15.8]{berberian}, has central cover $1$,
  and lies in $\masa$.
\end{proof}

\begin{lemma}\label{lem:KR6.9.22}
  Let $e_1$ be an abelian projection with $\cover{e_1}=1$ in an
  AW*-algebra $\alg$ of type $\typeone_n$ for $n$ finite. 
  Then there is a set of $n$ orthogonal equivalent projections with
  sum $1$ in $\alg$ containing $e_1$ (so that each is abelian in $\alg$),
  and $(1-e_1)\alg (1-e_1)$ is of type $\typeone_{n-1}$.
\end{lemma}
\begin{proof}
  Fix orthogonal equivalent abelian projections $f_1, \dots, f_n$ 
  with $\sum f_i = 1$. Then $e_1 \sim f_1$ ($\sim f_i$ for all $i$)
  by~\cite[Proposition~18.2(1)]{berberian}, and $1-e_1 \sim 1-f_1$
  by~\cite[Proposition~17.5]{berberian}.
  So $1-e_1$ is the sum of $n-1$ orthogonal equivalent abelian
  projections that are equivalent to $f_1 \sim e_1$ because the same is
  true for $1-f_1 = f_2 + \cdots + f_n$. The claim now follows.
\end{proof}

\begin{proposition}\label{prop:KR6.9.23}
  Let $\alg$ be an AW*-algebra of type $\typeone_n$ with $n$ finite, and
  let $\masa$ be a maximal abelian subalgebra.
  \begin{enumerate}[(a)]
  \item There is an orthogonal set $\{e_1, \dots, e_n\}$ of equivalent abelian
    projections in $\masa$ with sum $1$ with central cover $1$.
  \item $\masa$ contains $l$ orthogonal projections with sum $1$
    equivalent in $\alg$ if $n=lm$ (with $l$ and $m$ positive integers).
  \end{enumerate}
\end{proposition}
\begin{proof}
  Part (a) is proven by induction on $n$. If $n=1$, then $\alg$ is
  abelian, $\masa =\alg$, and $1$ is a projection in $\masa$ abelian in $\alg$ with
  $\cover{1}=1$. Moreover, $\masa$ is the centre of $\alg$. Suppose $n>1$ and our
  assertion is established when $\alg$ is of type $\typeone_k$ for
  $k<n$. Then $\alg$ has no infinite central
  summands. Lemma~\ref{lem:KR6.9.21}(b) applies, giving an abelian projection
  $e_1\in \masa$ with $\cover{e_1}=1$. It follows from
  Lemma~\ref{lem:KR6.9.22} that $(1-e_1)\alg (1-e_1)$ is of type
  $\typeone_{n-1}$, and $(1-e_1)\masa$ is a maximal abelian
  subalgebra. By the inductive hypothesis, $1-e_1$ is the sum of $n-1$
  projections $e_2,\ldots,e_n$ in $(1-e_1)\masa$ that are abelian in
  $(1-e_1)\alg (1-e_1)$ (and hence in $\alg$), and has central cover $1-e_1$
  in $(1-e_1)\alg (1-e_1)$. From Lemma~\ref{lem:centralcoverincorners}
  it follows that $1 = \cover{e_j}$ for $j \geq 2$, and since $\cover{e_1} = 1$
  as well we must have $e_i \sim e_j$ for all $j$ by~\cite[Proposition~18.1]{berberian}.

  For (b), set $f_j = \sum_{k=0}^{m-1} e_{j+kl}$ for
  $j=1,\ldots,l$. Then $f_1,\ldots,f_l$ are orthogonal projections in
  $\masa$ with sum $1$ equivalent in $\alg$.
\end{proof}

\section{Dimension theory}\label{sec:dimension}

Let $e$ be a properly infinite projection in an AW*-algebra $\alg$. We are going to define a
cardinal number $\d(e)$, that we think of as the ``dimension'' of
$e$. The goal of this section is to prove that $e \precsim f$ and
$\cover{e}=\cover{f}$ imply $\d(e)\leq \d(f)$. The next section
will prove the converse in a special case of interest.

Let $\Gamma(e)$ denote the set of all orthogonal families $\{e_i\}$ of
projections such that $e = \sum e_i$ and every $e_i \sim
e$. Lemma~\ref{lem:properlyinfinite} guarantees that $\Gamma(e)$ 
contains an infinite set.
If $\Lambda$ is a set of cardinals, we let $\supP \Lambda$ denote the
least cardinal that is strictly greater than every element of
$\Lambda$. Evidently $\supP \Lambda = \sup 
\{\alpha^+ \mid \alpha \in \Lambda\}$, where $\alpha^+$ denotes the successor
of a cardinal $\alpha$. 
 
\begin{definition}\label{def:tau}
  For a properly infinite projection $e$ in an AW*-algebra, define 
  \begin{align*}
    \d(e) &= \supP \{ \card I \mid \{e_i\}_{i \in I} \in \Gamma(e) \}, \\
    \dbar(e) &= \sup \{ \d(ze) \mid 0 < z \leq \cover{e} \mbox{ is central} \}.
  \end{align*}
  By convention, we agree that $\d(0) = \dbar(0) = 0$. (When the
  algebra $\alg$ in which $\d(e)$ and $\dbar(e)$ are computed needs to
  be emphasized, we will write $\d_{\alg}(e)$ and $\dbar_{\alg}(e)$.)
\end{definition}

As a basic example, suppose that $e \in B(H)$ is a projection on a Hilbert space $H$
whose range $e(H)$ is infinite dimensional. Then $e$ is properly infinite and $\d(e)$
is the successor cardinal of the dimension of $e(H)$, that is, $\d(e) = (\dim e(H))^+$.

The definition of $\d(e)$ uses successors because it is not clear
whether the supremum is achieved, \textit{i.e.}~whether there always exists a
family in $\Gamma(e)$ with cardinality $\sup \{ \card {I}
\mid \{e_i\}_{i \in I} \in \Gamma(e)\}$. 
If this supremum is indeed achieved for all properly infinite projections in all
AW*-algebras, then it would be more sensible to set $\d(e)$ equal to the
supremum $\sup\{ \card(I) \mid \{e_i\}_{i \in I} \in \Gamma(e)\}$;
all results about $\d(e)$ proved below would still hold.
The supremum is always achieved when
the cardinal $\sup \{ \card{I} \mid \{e_i\}_{i \in i} \in \Gamma(e)\}$
is not weakly inaccessible, and when $A$ is a von Neumann algebra; see
Appendix~\ref{app:supachieved}. 
We leave the general question open, and move on to basic results about $\d$. 

Notice that if $e$ is a properly infinite projection in an AW*-algebra $\alg$, then $\Gamma(e)$
and $\d(e)$ are the same whether ``computed'' in $\alg$ or $e \alg
e$. Thus $\d$ is invariant under passing to corners.  
Also, if $e \sim f$ in $\alg$ then $\d(e) = \d(f)$. 

\begin{lemma}\label{lem:alphasquared}
  Let $e$ be a projection in an AW*-algebra $\alg$.
  \begin{enumerate}[(a)]
  \item If $e=\sum_{i \in \alpha} e_i$ for an infinite cardinal $\alpha$, with $\{e_i\}$
    all nonzero and pairwise equivalent, then $e=\sum_{i \in \alpha} e'_i$ with $e'_i \sim e$. So $e$ is
    properly infinite with $\alpha<\d(e)$.
  \item If $e$ is properly infinite projection, and $\{e_i \mid i \in
    \alpha\}$ is an orthogonal set of projections with all $e_i \sim e$
    for a cardinal $\alpha<\d(e)$, then $\sum e_i \sim e$.
  \end{enumerate}
\end{lemma}
\begin{proof}
  For (a); since $\alpha$ is infinite, we have $\alpha^2=\alpha$
  (in cardinal arithmetic). So we can reindex $\{e_i \mid i \in
  \alpha\}$ as $\{e_{ij} \mid i,j \in \alpha\}$, and obtain
  $e=\sum_{i,j \in \alpha} e_{ij}$ with all $e_{ij}$ equivalent and
  orthogonal. Set $e'_i = \sum_{j \in \alpha} e_{ij}$. Then each
  $e'_i \sim e$, and $\sum_{i \in \alpha} e'_i =  e$. 
  
  We turn to (b). Because $\alpha < \d(e)$, there exists a set
  $\{f_i \mid i \in \alpha\} \in \Gamma(e)$. 
  Then $f_i \sim e \sim e_i$ for all $i$, so additivity of equivalence gives $\sum e_i \sim\
  \sum f_i = e$.
\end{proof}

\begin{lemma}\label{lem:downwardclosed}
  If $e$ is a properly infinite projection in an AW*-algebra $\alg$, then the set of cardinals
  $\{\card I \mid \{e_i\}_{i \in I} \in \Gamma(e)\}$ is downward-closed.
\end{lemma}
\begin{proof}
  Suppose that $\{e_i \mid i \in \beta \} \in \Gamma(e)$ for some cardinal $\beta$,
  and consider any cardinal $\alpha \leq \beta$. We will construct a set in
  $\Gamma(e)$ of cardinality $\alpha$. Write $\beta = \bigsqcup_{j \in \alpha}
  \beta_j$ as a disjoint union of $\alpha$-many subsets $\beta_j$ which each
  have cardinality $\beta$ (this is possible because $\alpha \cdot \beta = \beta$
  in cardinal arithmetic).
  For each $j \in \alpha$, let 
  $
  f_j = \sum_{i \in \beta_j} e_i
  $.
  By additivity of equivalence, $f_j \sim \sum_{i \in \beta} e_i = e$. Thus
  $\{f_j \mid j \in \alpha\} \in \Gamma(e)$.
\end{proof}

\begin{lemma}\label{lem:tau}
  Let $e$ be a projection in an AW*-algebra $\alg$. 
  \begin{enumerate}[(a)]
  \item If $e$ is properly infinite, then $\d(e) \leq \d(ze)$ for
    central projections $0 < z \leq \cover{e}$. 
  \item If $e=\sum e_i$ for an orthogonal set $\{e_i\}$ of properly
    infinite projections, then $e$ is properly infinite and $\d(e) \geq \min\{\d(e_i)\}$.
  \item If $e$ is properly infinite and $\cover{e} = \sum z_i$ for nonzero
    central projections $z_i$, then $\d(e) = \min\{\d(z_i e)\}$.
  \item If $e$ is properly infinite, then $\dbar(ze) \leq
    \dbar(e)$ for any nonzero central projection $z$. 
  \end{enumerate}
\end{lemma}
\begin{proof}
  Part (a) follows from the observation that if $\{e_i\} \in
  \Gamma(e)$, then $\{ze_i\} \in \Gamma(ze)$ for any nonzero central
  projection $z \leq \cover{e}$.

  For (b), fix an infinite cardinal $\alpha < \min\{\d(e_i)\}$; then for each
  $i$ there exists $\{e_{ij} \mid j \in \alpha\} \in \Gamma(e_i)$.
  For each $j \in \alpha$, define $e_j = \sum_i e_{ij}$. By additivity of equivalence,
  $e_j \sim \sum_i e_i = e$ for all $j$. Because $\sum e_j = e$, we find that $e$ is properly
  infinite with $\{e_j \mid j \in \alpha\} \in \Gamma(e)$ and thus $\alpha < \d(e)$. This
  demonstrates that $\d(e) \geq \min\{\d(e_i)\}$.

  Part (c) follows from (a) and (b).
  Part (d) follows by verifying the equations
  \begin{align*}
    \dbar(e) &= \sup \{ \d(ye) \mid 0 < y \leq \cover{e} \mbox{ is central} \}, \\
    \dbar(ze) &= \sup \{ \d(ye) \mid 0 < y \leq \cover{ze} = z\cover{e} \mbox{ is central} \}, 
  \end{align*}
  and noticing that the latter set over which the $\sup$ is quantified is a
  subset of the former. 
\end{proof}

Theorem~\ref{thm:cilin1} below partly justifies the intuition that $\d(e)$ measures a
``dimension'' of $e$.
If $e \leq f$ are properly infinite projections, then one might expect to have
$\d(e) \leq \d(f)$; this is true under the additional hypothesis that $e$ and $f$ have the
same central cover.
The proof requires transfinite repetition of the following construction.

\begin{lemma}\label{lem:eunderf}
  Let $p \leq q$ be projections in an AW*-algebra $\alg$, and suppose that $p$ is properly infinite.
  There exist projections $p' \in \alg$ and central $z \leq \cover{p}$ satisfying:
    \begin{itemize}
    \item $zp \sim zq$;
    \item $(\cover{p} - z) p \sim p' \leq (\cover{p} - z)(q - p)$ (in particular, $pp' = 0$ and $p' \leq q$);
    \item $\cover{p'} = \cover{p} - z$.
    \end{itemize}
\end{lemma}
\begin{proof}
  We may pass to the summand $\cover{p} \alg$ and assume that $\cover{p}
  = 1$. By generalized comparability, there exists a central
  projection $z$ such that $z(q-p) \precsim zp$ and $(1-z)p \precsim (1-z)(q-p)$.
  Because $p$ is properly infinite, we may write $p = p_1 + p_2$ for some projections
  $p_1 \sim p_2 \sim p$. Then $z(q-p) \precsim zp \sim zp_1$ and $zp \sim zp_2$. It follows
  that
  \[
  zp \leq zq = z(q-p) + zp \precsim zp_1 + zp_2 = zp,
  \]
  whence $zp \sim zq$. 
  Because $(1-z)p \precsim (1-z)(q-p)$, there exists a projection $p' \leq (1-z)(q-p) \leq q - p$
  such that $(1-z)p \sim p'$. 
  Also, $p' \sim (1-z)p$ means that $\cover{p'} = \cover{(1-z)p} = (1-z) \cover{p}
  = \cover{p} - z$.
\end{proof}

The proof below will regard the cardinal $\d(e)$ as an initial ordinal: the smallest ordinal
in its cardinality class.

\begin{theorem}\label{thm:cilin1}
  Let $e$ and $f$ be properly infinite projections in an AW*-algebra $\alg$. If 
  $\cover{e} = \cover{f}$ and $e \precsim f$, then $\d(e) \leq \d(f)$.
\end{theorem}
\begin{proof}
  Passing to the summand $\cover{e} \alg$ and replacing $e$ with an equivalent
  projection $e' \leq f$, we may assume that $\cover{e} = 1 = \cover{f}$ and $e \leq f$.
  We will build projections $z_\alpha$ and $e_\alpha$
  for ordinals $\alpha$ with the following properties: 
  \begin{enumerate}[(a)]
  \item $\{z_\alpha\}$ are central and orthogonal (and possibly zero);
  \item $\cover{e_\alpha} = 1 - \sum_{\beta\leq\alpha} z_\beta$ (so if $e_\alpha = 0$ then $1 = \sum_{\beta \leq \alpha} z_\beta$);
  \item $\{e_\alpha\}$ are orthogonal projections below $f$;
  \item if $z_\alpha > 0$ then $\d(z_\alpha f) \geq \d(e)$;
  \item if $e_\alpha > 0$ then it is properly infinite and $\d(e_\alpha) \geq \d(e)$;
  \end{enumerate}
  such that the process terminates exactly when $\sum z_\alpha =
  1$. Notice that if $\sum e_\alpha=f$, then we must have
  $e_{\alpha+1}=0$, so that the process terminates then by condition~(b).
  
  For $\alpha=0$, set $z_0=0$ and $e_0=e$. Now, suppose that
  $z_\beta$ and $e_\beta$ have already been constructed for
  all $\beta<\alpha$, and that $\sum_{\alpha < \beta} z_\beta < 1$.
  Notice by~(b) that the $\cover{e_\beta}$ form a decreasing chain
  and that 
  \[
    y := \bigwedge_{\beta < \alpha} \cover{e_\beta} = 1 - \sum_{\beta < \alpha} z_\beta.
  \]
  We are assuming that this central projection $y$ is nonzero.
  For each $\beta$, condition~(e) and $0 < y \leq \cover{e_\beta}$
  imply that $ye_\beta$ is properly infinite with $\d(ye_\beta) \geq \d(e_\beta)
  \geq \d(e)$.
  Set
  \[
    p= y \sum_{\beta<\alpha} e_\beta = \sum_{\beta < \alpha} ye_\beta.
  \]
  By Lemma~\ref{lem:tau}(b), $p$ is properly infinite and $\d(p) \geq
  \min\{\d(ye_\beta) \mid \beta<\alpha\} \geq \d(e)$. Furthermore,
  $y \leq \cover{e_\beta}$ for $\beta<\alpha$ by construction,
  so that $\cover{p} = \bigvee \cover{ye_\beta} = y$.
  Applying Lemma~\ref{lem:eunderf} to $p$ and $q=f$ now gives projections $z_\alpha = z$ and
  $e_\alpha=p'$ with the following properties.
  \begin{enumerate}[(a)]
    \item By construction, $z_\alpha$ is central with $z_\alpha \leq \cover{p}
    = 1 - \sum_{\beta < \alpha} z_\beta$.
    Therefore $z_\alpha \perp z_\beta$ for all
    $\beta<\alpha$, and $\{z_\beta \mid \beta \leq \alpha\}$ is orthogonal.
    \item We have $
      \cover{e_\alpha}
      = \cover{p}-z_\alpha 
      = (1-\sum_{\beta<\alpha} z_\beta) - z_\alpha 
      = 1-\sum_{\beta\leq\alpha} z_\beta$.
    \item Directly from Lemma~\ref{lem:eunderf} we have $e_\alpha \leq f - p$.
      So $e_\alpha \leq f$ and $e_\alpha \perp p$, which implies $e_\alpha \perp ye_\beta$
      for all $\beta<\alpha$. Because $\cover{e_\alpha} \leq \cover{p} = y$,
      this means that $e_\alpha \perp e_\beta$ for all $\beta < \alpha$.
      Hence $\{e_\beta \mid \beta \leq \alpha\}$ is orthogonal.
    \item Next, $z_\alpha$ is chosen so that $z_\alpha p \sim z_\alpha f$. Combined with
      $z_\alpha \leq \cover{p}$, we see that if $z_\alpha \neq 0$ then $z_\alpha f \sim
      z_\alpha p$ is properly infinite and $\d(z_\alpha f) = \d(z_\alpha p) \geq \d(p)
    \geq \d(e)$.
    \item Finally, assume $e_\alpha > 0$. The construction of $e_\alpha$ guarantees
        that $e_\alpha \sim \cover{e_\alpha} p$. Since $0 < \cover{e_\alpha} \leq \cover{p}$,
        this means that $e_\alpha$ is properly infinite and we have
        $\d(e_\alpha) = \d(\cover{e_\alpha} p) \geq \d(p) \geq \d(e)$.
  \end{enumerate}
  Transfinite induction now gives us the desired projections
  $\{z_\alpha\}$, $\{e_\alpha\}$.

  If there is a step $\alpha$ in the construction above for which $\sum_{\beta \leq \alpha}
  z_\beta = 1$, then by condition~(d) we have that $f = \sum z_\beta f$ is a sum of properly
  infinite projections with $\d(z_\beta f) \geq \d(e)$ (ignoring those $z_\beta$ which are zero).
  In this case, we conclude from Lemma~\ref{lem:tau}(b) that $\d(f) \geq \d(e)$.
  
  Finally, suppose that $\sum_{\beta \leq \alpha} z_\beta < 1$ at every step $\alpha$.
  Then condition~(b) guarantees that each $e_\alpha$ is nonzero. So the projections
  $\sum_{\beta \leq \alpha} e_\beta$ form a strictly increasing sequence below $f$.
  This chain cannot increase without bound (for instance, it is bounded by $\card(\Proj(fAf))^+$),
  so there exists $\alpha$ such that $\sum_{\beta \leq \alpha} e_\beta = f$. From
  condition~(e) and Lemma~\ref{lem:tau}(b) we once again conclude that $\d(f) \geq \d(e)$.
 \end{proof}

As an easy consequence, we see that $\dbar$ behaves in the same way.

\section{Equidimensional projections}\label{sec:equidimensional}

The hypothesis in Theorem~\ref{thm:cilin1} that the projections $e$ and $f$
satisfy $\cover{e} = \cover{f}$ cannot be removed. For instance, let $H$ and $K$
be infinite dimensional Hilbert spaces with $\dim(H) < \dim(K)$ and consider the
AW*-algebra $A = B(H) \oplus B(K)$.
One can readily compute that $\d((1_H,1_K)) = \dim(H)^+$ and $\d((0,1_K)) =
\dim(K)^+$. Thus $(0,1) < (1,1)$ but $\d((0,1)) > \d((1,1))$.
The issue is that images of the projection $(1,1)$ in the two central
summands $(1,0)A$ and $(0,1)A$ have different dimensions.

It will prove fruitful to focus on so-called equidimensional projections: those
projections for which the above pathology does not occur. We will
show that such projections are equivalent precisely when they have
the same central cover and dimension. Moreover, we will prove that any
properly infinite projection is a sum of equidimensional ones. 

\begin{definition}
  A properly infinite projection $e$ in an AW*-algebra $\alg$ will be called \emph{equidimensional}
  if $\d(ze) = \d(e)$ for every nonzero central projection $z \leq
  \cover{e}$. The AW*-algebra $\alg$ is called \emph{equidimensional} when $1_{\alg} $
  is equidimensional. We say that $e$ is \emph{$\alpha$-equidimensional} for
  a cardinal $\alpha$ if $e$ is equidimensional with $\d(e) = \alpha$. 
  By convention, we will also agree that $0 \in A$ is a 0-equidimensional projection.
\end{definition}


It is straightforward to see that a properly infinite projection $e$
in an AW*-algebra $\alg$ is equidimensional if and only if there exists an infinite cardinal
$\alpha$ such that, for every central projection $z \in \alg$, either $ze =0$
or $\d(ze) = \alpha$.

\begin{lemma}\label{lem:equidimincorners}
  Let $e$ be a properly infinite projection in an AW*-algebra $\alg$.
  \begin{enumerate}[(a)]
  \item $e\alg e$ is equidimensional if and only if $e$ is equidimensional.
  \item If $e$ is equidimensional and $e \sim f$, then $f$ is
    equidimensional. 
  \end{enumerate}
  Hence $e\alg e$ is equidimensional when $\alg$ is equidimensional and $e \sim 1$.
\end{lemma}
\begin{proof}
  For (a), let $e$ be equidimensional and let $z$ central in $e\alg e$. Then
  $z = he$ for some central projection $h$ of $\alg$  by
  Lemma~\ref{lem:centralcoverincorners}. 
  Since $e$ is equidimensional in $\alg$, 
  \[
    \d_{e\alg e}(z) = \d_{e \alg e} (he) = \d_{\alg} (he) = \d_{\alg} (e) = \d_{e\alg e}(e).
  \]
  Conversely, assume that $e\alg e$ is equidimensional, and let
  $h \in \alg$ be a central projection. Then $he$ is central in $e\alg e$, so
  \[
    \d_{\alg} (he) = \d_{e\alg e}(he) = \d_{e\alg e}(e) = \d_{\alg} (e).
  \]

  For (b), notice that whenever $z \in \alg$ is a central projection,
  $ze \sim zf$ and thus $\d(ze) = \d(zf)$. The statement clearly follows.
\end{proof}

The following theorem establishes another desired property of a ``dimension''
measure. If the dimension of a projection $e$ is strictly less than the dimension
of a projection $f$, intuition developed in $B(H)$ might lead one to expect that
$e \prec f$.
The example $A = B(H) \oplus B(K)$ with $\dim(H) < \dim(K)$ infinite again shows
that this cannot hold in full generality: fixing any orthogonal projection of $K$
onto a subspace of dimension $\dim(H)$, we have $(1,p) < (1,1)$ and even
$\cover{(1,p)} = \cover{(1,1)}$, but $\d((1,p)) = \d((1,1))$.
As mentioned above, the key assumption that both $e$ and $f$ be equidimensional
makes the intuitive idea true.
The proof below basically uses the same transfinite construction as the proof of
Theorem~\ref{thm:cilin1}, but with different termination conditions. For the sake
of readability, we write it out in full.
 
\begin{theorem}\label{thm:cilin2}
  Let $e$ and $f$ be properly infinite, equidimensional projections in
  an AW*-algebra.
  If $\cover{e}=\cover{f}$ and $e \prec f$, then $\d(e)<\d(f)$.
\end{theorem}
\begin{proof}
  Passing to the summand $\cover{e} \alg$ and replacing $e$ with an equivalent
  projection below $f$, we may assume that $\cover{e} = 1 = \cover{f}$, $e \leq f$,
  and $e \not\sim f$.  
  We will build projections $z_\alpha$ and $e_\alpha$
  for ordinals $\alpha<\d(e)$ (regarding $\d(e)$ as an initial ordinal) with the following
  properties: 
  \begin{enumerate}[(a)]
  \item $\{z_\alpha\}$ are central and orthogonal (and possibly zero);
  \item $\cover{e_\alpha} = 1 - \sum_{\beta\leq\alpha} z_\beta$;
  \item $\{e_\alpha\}$ are orthogonal projections below $f$; 
  \item $e_\alpha \sim \cover{e_\alpha} e$;
  \item $z_\alpha e \sim z_\alpha f$.
  \end{enumerate}
  For $\alpha=0$, set $z_0=0$ and $e_0=e$. Now, suppose that
  $z_\beta$ and $e_\beta$ have already been constructed for
  all $\beta<\alpha$. 
  Notice by (b) that the $\cover{e_\beta}$ form a decreasing sequence
  and that 
  \[
    y := \bigwedge_{\beta < \alpha} \cover{e_\beta} = 1 - \sum_{\beta < \alpha} z_\beta.
  \]
  Condition (e) and $e \not\sim f$ guarantee that this central projection $y$ is nonzero.
  For each $\beta$, condition (d) together with $y \leq \cover{e_\beta}$
  give $ye_\beta \sim ye$; notice $ye \neq 0$ because $\cover{e} = 1$.
  Set
  \[
    p= y \sum_{\beta<\alpha} e_\beta = \sum_{\beta < \alpha} ye_\beta.
  \]
  Because $\card \alpha < \d(e)$ (as $\alpha$ is strictly below the initial
  ordinal $\d(e)$), Lemma~\ref{lem:alphasquared}(b) implies that $p
  \sim ye$. So
  $p$ is properly infinite and $\cover{p} = \cover{ye} = y = 1 - \sum_{\beta < \alpha} z_\beta$.
  Applying Lemma~\ref{lem:eunderf} to $p$ and $q=f$ now gives projections $z_\alpha = z$ and
  $e_\alpha=p'$ with the following properties.
  \begin{itemize}
    \item[(a)--(c)] These follow just as conditions (a)--(c) in the proof of Theorem~\ref{thm:cilin1}.
    \item[(d)] Next, the construction of $e_\alpha$ along with $\cover{e_\alpha} \leq \cover{p} = y$
      and $p \sim ye$ shows that $e_\alpha \sim \cover{e_\alpha} p \sim \cover{e_\alpha} e$.
    \item[(e)] Finally, $z_\alpha$ is chosen so that $z_\alpha p \sim z_\alpha f$. Combined with
    $z_\alpha \leq y = \cover{p}$ and $p \sim ye$, we have $z_\alpha f \sim z_\alpha p \sim z_\alpha e$.
  \end{itemize}
  Transfinite induction now gives us the desired projections
  $z_\alpha,e_\alpha$ for $\alpha<\d(e)$.

  Set $z = \sum_{\alpha<\d(e)} z_\alpha$, so that $zf \sim ze$ by~(e) and
  $1-z = \bigwedge_{\alpha < \d(e)} \cover{e_\alpha}$ by~(b).
  Since $e \not\sim f$, we must have $1-z>0$.
  Also~(d) implies that $(1-z)e_\alpha \sim (1-z)e > 0$ for all $\alpha<\d(e)$.
  Furthermore, as each $\cover{(1-z)e_\alpha} = (1-z) = \cover{(1-z)f}$, we
  have $\cover{\sum (1-z)e_\alpha} = 1-z = \cover{(1-z)f}$. Therefore 
  \begin{align*}
    \d(e)
    & < \d \left( \sum\nolimits_{\alpha<\d(e)} (1-z)e_\alpha \right) 
       \eqcomment{by Lemma~\ref{lem:alphasquared}(a)} \\
    & \leq \d((1-z)f ) 
       \eqcomment{by Theorem~\ref{thm:cilin1}} \\
    & =\d(f),
       \eqcomment{$f$ is equidimensional}
  \end{align*}
  as desired.
\end{proof}

\begin{corollary}\label{cor:cilin}
  Let $e$ and $f$ be properly infinite, equidimensional projections in
  an AW*-algebra.
  Then $e \precsim f$ if and only if $\cover{e} \leq \cover{f}$ and
  $\d(e) \leq \d(f)$. Therefore $e \sim f$ if and only if
  $\cover{e} = \cover{f}$ and $\d(e)=\d(f)$.
\end{corollary}
\begin{proof}
  If $\cover{e} \leq \cover{f}$ then $\cover{e}f$ is equidimensional
  and $\d(\cover{e}f) = \d(f)$. Replacing $f$ by $\cover{e}f$, we may
  assume $\cover{e} = \cover{f}$ and prove that $e \precsim f$ if and only if
  $\d(e) \leq \d(f)$.
  
  One direction is just Theorem~\ref{thm:cilin1}. For the
  other, suppose that $\d(e) \leq \d(f)$.
  The comparison theorem gives us a central projection $z$
  satisfying $ze \precsim zf$, and $(1-z) e \succ (1-z) f$, and
  $1-z \leq \cover{e} = \cover{f}$.  If $z<1$, then $(1-z)e$ and
  $(1-z)f$ are nonzero and properly infinite, so by equidimensionality
  and Theorem~\ref{thm:cilin2} we have $\d(e) = \d((1-z) e) > \d((1-z) f) = \d(f)$,
  which contradicts the assumption $\d(e) \leq \d(f)$. Thus $z = 1$ and $e \precsim f$. 
\end{proof}

In order to make use of Corollary~\ref{cor:cilin} in an arbitrary AW*-algebra, there
must be a rich supply of equidimensional projections. This will be demonstrated in
Theorem~\ref{thm:supofequidims}, after the following preparatory lemma.

\begin{lemma}\label{lem:equidimensionalsexist}
  Let $e$ be a properly infinite projection in an AW*-algebra. Then:
  \begin{enumerate}[(a)]
  \item $e$ is equidimensional if $\dbar(ze)=\dbar(e)$ for all
    central projections $0< z \leq \cover{e}$;
  \item there exists a nonzero central projection $z \leq \cover{e}$ making
    $ze$ equidimensional.
  \end{enumerate}
\end{lemma}
\begin{proof}
  To prove (a), suppose towards a contradiction that $e$ is not
  equidimensional. Then $\d(e) < \d(z_0 e)$ for some nonzero
  central $z_0 \leq \cover{e}$ by Lemma~\ref{lem:tau}(a). Zorn's lemma
  allows us to extend $\{z_0\}$ to a maximal set $\{z_i\}$ of orthogonal
  nonzero projections such that $\d(e) < \d(z_i e)$. Because
  $\d(e) < \min\{\d(z_i e)\}$, it follows from Lemma~\ref{lem:tau}(c)
  that $z = \cover{e} - \sum z_i$ is nonzero. Applying that same lemma to the
  set of projections $\{z_i\} \cup \{z\}$ with sum $\cover{e}$, we must have
  $\d(e) = \d(ze)$. Using the hypothesis, 
  \[
    \dbar(ze) = \dbar(e) = \dbar(z_0 e).
  \]
  Since $\d(e) < \d(z_0 e) \leq \dbar(z_0 e) = \dbar(ze)$, by definition
  of $\dbar(ze)$ there is a nonzero central projection $y \leq \cover{ze}
  = z$ such that $\d(ye) = \d(yze) > \d(e)$. But this contradicts the
  maximality of $\{z_i\}$. We conclude that $e$ must be equidimensional.

  As for (b): by well-ordering, there is a nonzero central projection $z \leq
  \cover{e}$ minimizing $\dbar(ze)$. Let $y \leq \cover{ze}=z$ be a nonzero
  central projection. Then it follows from Lemma~\ref{lem:tau}(d) that
  $\dbar(ye)\leq\dbar(ze)$. Therefore $\dbar(ye)=\dbar(ze)$
  by minimality of $\dbar(ze)$. Hence $ze$ is equidimensional by (a).
\end{proof}

\begin{theorem}\label{thm:supofequidims}
  Let $e$ be a properly infinite projection in an AW*-algebra $\alg$.
  \begin{enumerate}[(a)]
   \item Each infinite cardinal $\alpha \leq \dbar(e)$ allows a
     largest central projection $z_\alpha \leq \cover{e}$ 
   such that $z_\alpha e$ is $\alpha$-equidimensional. These projections are
   orthogonal for distinct $\alpha$.
   \item Letting $\alpha$ range as above, we have $\cover{e} = \sum
     z_\alpha$. 
   \end{enumerate}
  Thus $e = \sum z_\alpha e$ is a sum of equidimensional projections.
\end{theorem}
\begin{proof}
  Fix $\alpha$ as in part~(a). 
  Zorn's lemma produces a maximal orthogonal family $\{z_i\}$ of
  nonzero central projections $z_i \leq \cover{e}$ where each $z_i e$ 
  is $\alpha$-equidimensional. Set $z_\alpha = \sum z_i$. It is
  straightforward to verify that
  $z_\alpha e = \sum z_i e$ is $\alpha$-equidimensional using Lemma~\ref{lem:tau}(c).
  Furthermore, if $z \leq \cover{e}$ is central and $ze$ is
  $\alpha$-equidimensional, then the projection $z(\cover{e}-z_\alpha)$
  is central and orthogonal to all $\{z_i\}$. If it is nonzero then
  $z(\cover{e} - z_\alpha)e$ is $\alpha$-equidimensional. Maximality of
  $\{z_i\}$ thus requires $z(\cover{e} - z_\alpha) = 0$, or
  $z \leq z_\alpha$. 
  For cardinals $\alpha$ and $\beta$, if $z_\alpha z_\beta$ is nonzero
  then $\alpha = \d(z_\alpha z_\beta e) = \beta$. Thus $\alpha \neq \beta$
  implies $z_\alpha z_\beta = 0$. 

  For (b), assume for contradiction that $y = \cover{e} - \sum z_\alpha > 0$.
  Then $0 < ye \leq e$ with $\cover{ye} = y$. Lemma~\ref{lem:equidimensionalsexist}(b)
  provides a nonzero projection $z \leq y$ such that $ze$ is equidimensional, say
  with $\d(ze) = \beta$. But then $z \leq z_\beta \leq \sum z_\alpha$, contradicting
  that $0 < z \leq y = \cover{e} - \sum z_\alpha$. So we must have
  $\cover{e} = \sum z_\alpha$.
\end{proof}

We conclude this section by bringing our treatment more in line
with the notions of dimension in the
literature~\cite{feldman:dimension,tomiyama:dimension,goodearlwehrung}. 
Such notions are traditionally defined in terms of $\Spec(Z(\alg))$, the Gelfand
spectrum of the centre of an AW*-algebra $\alg$, rather than using central projections.
We write $\varphi$ for the canonical *-isomorphism from $Z(\alg)$
to the algebra of continuous complex-valued functions on
$\Spec(Z(A))$. Given a properly infinite projection $e \in \alg$, we
define a function $\D[e]$ from $\Spec(Z(\alg))$ to the cardinals as
follows. Let $\{z_\alpha\}$ be the family provided by the previous
theorem, with the addition of $z_0 = 1 - \cover{e}$;
then $\supp(\varphi(z_\alpha))$ are disjoint
clopens that cover $\Spec(Z(A))$ since $1 = \sum z_\alpha$.
Therefore the function from $\bigsqcup_\alpha \supp(\varphi(z_\alpha))$
to the cardinals, mapping $\supp(\varphi(z_\alpha))$ to $\alpha$,
is continuous when we put the order topology on the
cardinals~\cite[Section~X.9]{birkhoff:lattices}. Because $\sum
z_\alpha = \cover{e}$, this function is defined on a dense subset, and
hence extends to a continuous function $\D[e]$ from all of
$\Spec(Z(\alg))$ to the cardinals~\cite[Lemma~5]{tomiyama:dimension}.
It follows from the previous theorem that
\[
  \D[e](t) = \d(ze) \qquad \text{ for } t \in \supp(\varphi(z))
\]
if $ze$ is equidimensional. Write $\D[e] \leq \D[f]$ to mean that
$\D[e](t) \leq \D[f](t)$ for all $t$.
We show that properly infinite projections $e,f \in A$ are comparable
if and only if the functions $\D[e]$ and $\D[f]$ are.
Using the methods developed above, we reduce the problem to
a test of the dimension of equidimensional summands.

\begin{proposition}\label{cor:tauze}
   For properly infinite projections $e$ and $f$ in an AW*-algebra $\alg$, $e \precsim f$ if and only if $\D[e] \leq \D[f]$. 
\end{proposition}
\begin{proof}
  First, we claim that $\D[e] \leq \D[f]$ if and only if this holds on
  a dense subset. An inequality $\alpha \leq \beta$ of cardinals
  holds precisely when the equality $\beta = \max\{\alpha,\beta\} =
  \alpha \beta$ holds. Recall that a net $\{\beta_i\}$ converges to
  $\beta$ in the order topology when there are a net $\{\alpha_i\}$
  increasing to $\beta$ and a net $\{\gamma_i\}$ decreasing to
  $\beta$ such that $\alpha_i \leq \beta_i \leq \gamma_i$. It clearly makes
  cardinal multiplication continuous, and the claim follows. 

  Let $1 = \sum x_\alpha = \sum y_\beta$ be central decompositions as
  in the above discussion, so that
  $x_\alpha e$ is $\alpha$-equidimensional and $y_\beta f$ is $\beta$-equidimensional.
  Since $1 = \sum x_\alpha y_\beta$ as well, $e \precsim f$ if and only if $x_\alpha y_\beta e
  \precsim x_\alpha y_\beta f$ for all $\alpha$ and $\beta$. Furthermore, the subsets
  $K_\alpha = \supp(\varphi(x_\alpha))$ and $L_\beta = \supp(\varphi(y_\beta))$ are
  (cl)open in $X = \Spec(Z(A))$, and $\bigcup K_\alpha$, $\bigcup L_\beta$ are open and
  dense in $X$, so $\bigcup (K_\alpha \cap L_\beta) = (\bigcup K_\alpha) \cap (\bigcup L_\beta)$
  is again dense in $X$. 

  Thus it suffices to show that $x_\alpha y_\beta e \precsim x_\alpha y_\beta f$ if and only if
  $\D[e](t) \leq \D[f](t)$ for all $t \in K_\alpha \cap L_\beta$. 
  Notice that $K_\alpha \cap L_\beta = \supp(\varphi(x_\alpha y_\beta))$. We may restrict
  to the case where $x_\alpha y_\beta > 0$, whence $K_\alpha \cap L_\beta \neq \emptyset$.
  In this case, $x_\alpha e$ is $\alpha$-equidimensional and $y_\beta f$ is $\beta$-equidimensional.
  Furthermore, if $t \in K_\alpha \cap L_\beta$, then $\D[e](t) = \d(x_\alpha y_\beta e) = \alpha$
  and $\D[f](t) = \d(x_\alpha y_\beta f) = \beta$. 
  So by Corollary~\ref{cor:cilin}, $x_\alpha y_\beta e \precsim x_\alpha y_\beta f$ if and only if
  $\alpha \leq \beta$, if and only if $\D[e](t) \leq \D[f](t)$ for all $t \in K_\alpha \cap L_\beta$.
\end{proof}

Using the known dimension theory of finite
projections in AW*-algebras~\cite[Chapter~6]{berberian} and
Lemma~\ref{lem:infiniteincorners}(b), the definition of $\D[e]$ can be extended
to arbitrary projections $e$, still satisfying the property
of the previous corollary, as in~\cite{tomiyama:dimension,goodearlwehrung}.

\section{Relative comparison for AW*-algebras of infinite type}
\label{sec:comparison:infinite}

Using the results about equidimensional projections, this section
carries out the relative comparison theory for a maximal abelian subalgebra
$\masa$ of a properly infinite AW*-algebra $\alg$, as Section~\ref{sec:comparison:finite}
did for finite algebras. Once again, whenever we mention without specification
concepts such as $\sim$, $\d$, finite, infinite, abelian, equidimensional, or
central cover, we mean the corresponding concepts in $\alg$ (and not in $\masa$). 
The results below are inspired by Kadison's~\cite{kadison:diagonalization},
but are suitably adapted for algebras that need not be countably decomposable.
Throughout this section, we will freely and repeatedly apply Lemmas~\ref{lem:corners}(c)
and~\ref{lem:infiniteincorners}.

We start by considering types $\typetwo_\infty$ and $\typethree$.
The key application of the dimension theory developed in the previous two
sections occurs in the proof of the following.

\begin{proposition}\label{prop:infinitemasa}
   Let $\alg$ be an AW*-algebra with a maximal abelian subalgebra $\masa$ all of
   whose nonzero projections are properly infinite.
   \begin{enumerate}[(a)]
   \item If $\alg$ is nonzero, there are projections $0 < e \leq p$ in $\masa$ satisfying $e \sim p \sim p-e$. 
   \item There is a projection $e$ in $\masa$ with $e \sim 1 \sim 1-e$.
   \end{enumerate}
\end{proposition}
\begin{proof}
  For (a),
  choose a nonzero $p \in \Proj(\masa)$ such that $\dbar_{\alg}(p) \leq \dbar_{\alg}(f)$
  for all $f \in \Proj(\masa)$; this can be done by well-ordering. 
  Because $\d$ is invariant under passing to corners, $\dbar_{p \alg p}(p)$ is
  also minimal, allowing us to drop the subscript.
  It follows from minimality of $\dbar(p)$ and
  Lemma~\ref{lem:tau}(d) that $\dbar(p)=\dbar(zp)$ for all nonzero
  central projections $z \leq \cover{p}$.
  Hence Lemma~\ref{lem:equidimensionalsexist}(a) guarantees that $p$ is equidimensional.
  Next, Lemma~\ref{lem:KR6.9.19} provides a
  projection $e$ in $p\masa$ with $\cover[p\alg p]{e} = p =
  \cover[p\alg p]{p-e}$. 
  In particular, $e$, $p$, and $p - e$ have the same central cover in
  $p \alg p$, and hence  by Lemma~\ref{lem:centralcoverincorners}
  also in $\alg$.
  If $z \leq \cover[p\alg p]{e}$ is a nonzero projection in $Z(p\alg
  p)=pZ(\alg)$, then
  \[
    \dbar(ze) \leq \dbar(e) \leq \dbar(p) \leq \dbar(ze)
  \]
  by, respectively, Lemma~\ref{lem:tau}(d), Theorem~\ref{thm:cilin1},
  and minimality of $\dbar(p)$. The same inequalities with $e$
  replaced by $p - e$ hold, so $\dbar(e) = \dbar(p) = \dbar(p-e)$, and
  $e$ and $p-e$ are equidimensional.
  Thus $\d(e) = \d(p) = \d(p-e)$.
  Now $e$, $p$, and $p - e$ are equivalent by Corollary~\ref{cor:cilin}.
  
  Proceeding to (b),
  Zorn's lemma produces a maximal set $\{p_i\}$ of orthogonal nonzero
  projections in $\masa$ such that there exist projections $\{e_i\} \subseteq \masa$ with
  $e_i \leq p_i$ and $e_i \sim p_i \sim p_i - e_i$ for all
  $i$. Assume, towards a contradiction, that $\sum p_i \neq 1$; then
  $s = 1-\sum p_i \in \masa$ is nonzero. 
  By assumption, $s$ is properly infinite. Projections in $s\masa s$ are
  properly infinite in $s\alg s$ by Lemma~\ref{lem:infiniteincorners}.
  So part~(a) applies to $s\alg s$ and its maximal
  abelian subalgebra $s\masa$, giving nonzero projections $e \leq p$ in
  $s\masa$ with $e \sim p \sim p-e$. Thus we may enlarge $\{p_i\}$ with
  $p$, contradicting maximality.  

  Hence $\sum p_i = 1$. Define $e = \sum e_i$, so that $1-e = \sum (p_i - e_i)$.
  Then $e \in \masa$, and additivity of equivalence provides $e \sim 1-e \sim \sum p_i = 1$
  as desired.
\end{proof}

\begin{lemma}\label{lem:supoffinites}
  If $\masa$  is a maximal abelian subalgebra in an AW*-algebra \alg , and $$e=\bigvee \{ f \in
  \Proj(\masa) \mid f \text{ is finite (in }\alg )\},$$ then projections in
  $(1-e)\masa$ are properly infinite. 
  If $\{f_i\}$ is a maximal orthogonal family of nonzero finite projections in $C$,
  then $e = \sum f_i$.
\end{lemma}
\begin{proof}
  Let $p \in (1-e)\masa$. Then $p \in \Proj(\masa )$ with $p \perp e$. 
  So for any central projection $z \in \alg$ such that $zp > 0$, also $\masa 
  \ni zp \perp e$, making $zp$ infinite by choice of $e$. Hence $p$ is properly infinite
  by Lemma~\ref{lem:properlyinfinite}.
  
  Let $\{f_i\}$ be a maximal orthogonal family of nonzero projections in $C$ that
  are finite; such a family exists by Zorn's lemma. Clearly $\sum f_i
  \leq e$. If $f \in C$ is any finite projection, then $f(1 - \sum f_i)$ is both
  finite and orthogonal to each $f_i$. By maximality, this product is zero,
  so $f \leq \sum f_i$. By definition of $e$, this means $e \leq \sum f_i$.
\end{proof}

\begin{lemma}\label{lem:KR6.9.4}
  Let $\alg$ be an AW*-algebra and $e, f \in \Proj(\alg)$.
  \begin{enumerate}[(a)]
  \item If $p,q \in \Proj(\alg )$ satisfy $e \precsim p \perp q
    \succsim f$, then $e \vee f \precsim p+q$.
  \item If $e$ is properly infinite and $f \precsim e$, then $e \sim e \vee f$.
  \item If $e$ is properly infinite, $f$ is finite, and $\cover{f} \leq \cover{e}$,
    then $e \vee f \sim e$.
  \end{enumerate}
\end{lemma}
\begin{proof}
  By~\cite[Theorem~13.1]{berberian}, $(e \vee f) - f \sim e - e \wedge f \precsim p$.
  Since $((e \vee f) - f)f = 0$ and $f \precsim q$, we have $e \vee f
  = (e \vee f) - f + f \precsim p+q$, establishing (a).

  We turn to (b). As $e$ is properly infinite, there is a projection $g \in \alg$ with
  $g<e$ and $e \sim g \sim e-g$.
  Then $f \precsim e \sim e-g$. Part (a) implies that $e \vee f \precsim g+e - g = e$.
  Since $e \leq e \vee f$, we have $e \sim e \vee f$ from Schr\"{o}der-Bernstein. 
  
  Toward (c), assume for contradiction that $f \not\precsim e$. 
  The comparison theorem now gives a nonzero
  central projection $z \leq \cover{f}$ with $ze \prec zf$. Because $z \leq \cover{f}
  \leq \cover{e}$, it follows from Lemma~\ref{lem:properlyinfinite} that
  $ze$ is (properly) infinite. This contradicts finiteness of $zf$, so
  $f \precsim e$. From part~(b) we conclude that $e \vee f \sim e$.
\end{proof}

\begin{proposition}\label{prop:infinitenotypeone}
  Let $\alg$ be a properly infinite AW*-algebra without central summands
  of type I, and let $\masa$ be a maximal abelian subalgebra of $A$. Then
  there exists a projection $e \in \masa$ such that $e \sim 1 \sim 1-e$.
\end{proposition}

\begin{proof}
  First we will produce $e \in C$ such that $e \sim 1-e$. Use Zorn's
  lemma to produce a maximal family $\{f_i\}$ of projections in $\masa$
  that are finite in $\alg$, and set $f=\sum f_i$. By Proposition~\ref{prop:KR6.9.27},
  there exist projections $e_{i1} \sim e_{i2}$ for all $i$ such that $f_i
  = e_{i1} + e_{i2}$. By Lemma~\ref{lem:supoffinites}, the projections
  of $(1-f)C$ are properly infinite. As $(1-f)C$ is a maximal abelian subalgebra
  of $(1-f)A(1-f)$, Proposition~\ref{prop:infinitemasa}
  provides projections $e'_1 \sim e'_2$ such that $(1-f) = e'_1 + e'_2$.
  Thus for $j = 1,2$, the projections $e_j = e'_j + \sum_i e_{ij}$ satisfy
  $e_1 \sim e_2$ and $1 = e_1 + e_2$. So we may take $e = e_1$.
  
  It remains to show that $e_1 \sim 1 \sim e_2$. Let $z$ be any central
  projection of $\alg$ such that $ze_1$ is finite; then $ze_2 \sim ze_1$ is
  finite, so that $z = ze_1 + ze_2$ is finite. But $\alg$ is properly infinite,
  so $z$ must be zero. Thus $e_1$ is properly infinite by
  Lemma~\ref{lem:properlyinfinite}(d). Since $e_2 \sim e_1$, it
  follows from Lemma~\ref{lem:KR6.9.4}(b) that $e_1 \sim e_1 + e_2 = 1$,
  so that $e_2 \sim e_1 \sim 1$ as desired.
\end{proof}

Next, we turn to AW*-algebras of type $\typeone_\infty$. 

\begin{lemma}\label{lem:KR6.9.24}
  Let $\alg$ be a nonzero AW*-algebra of type $\typeone_\infty$, and let
  $\masa$ be a maximal abelian subalgebra in which $1$ is the supremum
  of projections in $\masa$ finite in $\alg$.
  \begin{enumerate}[(a)]
  \item $\masa$ has a projection finite in $\alg$ with central cover $1$.
  \item $\masa$ has a projection abelian in $\alg$ with central cover $1$.
  \item Some nonzero central projection $z \in \alg$ is the sum of infinitely many 
    orthogonal equivalent projections in $\masa$.
  \item There is a projection $e \in \masa$ such that $e \sim 1 \sim 1-e$.
  \end{enumerate}
\end{lemma}
\begin{proof}
  For (a), let $\{f_j\}$ be a family of projections in $\masa$ finite in $\alg$
  and maximal with respect to the property that $\{\cover{f_j}\}$ is
  orthogonal. If $z=\sum \cover{f_j}$ and $z<1$, then $1-z$ is a nonzero
  projection in $\masa$. If $1-z$ is orthogonal to all finite projections
  of $\alg$ in $\masa$, the supremum of these finite projections is not $1$,
  contradicting the assumption. Thus there is a projection $f_0 \in \masa$
  finite in $\alg$ with $f_0(1-z)>0$. But then we may enlarge the family
  $\{f_j\}$ with $f_0(1-z)$, contradicting maximality. 
  Then $f=\sum f_j$ is a projection in $\masa$ finite with central
  cover $1$ in $\alg$~\cite[Proposition~15.8]{berberian}. 

  Towards (b), $f\alg f$ is an AW*-algebra of type $\typeone$, and $f\masa$ is a
  maximal abelian subalgebra. From Lemma~\ref{lem:KR6.9.21}(b), $f\masa$
  contains a projection $e_0$ abelian in $f\alg f$ (and hence in $\alg$) with
  $\cover[f\alg f]{e_0}=f$. Since $\cover[\alg]{e_0} \geq \cover[f \alg f]{e_0}
  = f$ 
  and $\cover[\alg ]{f}=1$,
  also $\cover[\alg ]{e_0}=1$.  Thus $e_0 \in \masa$ is a projection abelian
  with central cover $1$ in $\alg$.

  For (c), let $\{e_i\}$ be a maximal orthogonal family of projections
  in $\masa$ that are abelian with central cover $1$ in $\alg$, and set
  $e=\sum e_i$. By~\cite[Proposition~18.1]{berberian}, the $e_i$ are
  pairwise equivalent.
  If $e<1$, then $(1-e)\alg (1-e)$ is an AW*-algebra of type $\typeone$
  in which $(1-e)\masa$ is a maximal abelian subalgebra. Moreover,
  $(1-e)$ is the supremum of projections in $(1-e)\masa$ finite in
  $(1-e)\alg (1-e)$. From part (b), $(1-e)\masa$ contains a projection
  $e_1$ abelian with central cover $1-e$ in $(1-e)\alg (1-e)$. It follows
  that $e_1$ is abelian with central cover $\cover{1-e}$ in $\alg$. If
  $\cover{1-e}=1$, we can adjoin $e_1$ to $\{e_i\}$ contradicting
  maximality. Thus $z=1-\cover{1-e}$ is nonzero. Now $z(1-e)=0$,
  so that $z = ze = \sum ze_i$ and $\{ze_i\}$ is a family of orthogonal
  equivalent projections in $\masa$ with sum $z$. Because $\alg$ is
  properly infinite and the $ze_i$ are abelian, $z$ is infinite and
  $\{ze_i\}$ cannot be a finite set; see~\cite[Theorem~17.3]{berberian}.

  For (d), let $\{z_j \mid j \in \alpha \}$ be a maximal orthogonal
  family of central projections in $\alg$ each with the property of $z$
  from (c). 
  If $0<1-\sum z_j =: z_0$, then $z_0\alg$ is an AW*-algebra of type
  $\typeone_\infty$ and $z_0\masa$ is a maximal abelian subalgebra with
  the property that $z_0$ is the supremum of projections in $z_0\masa$
  finite in $z_0\alg$. Part~(c) provides a nonzero central projection
  $z_1$ in $z_0\alg$ that is the sum of infinitely many orthogonal
  equivalent projections in $z_0\masa$. Adjoining $z_1$ to $\{z_j\}$ produces a family contradicting
  maximality of $\{z_j\}$. Hence $\sum z_j=1$.

  For each $z_j$ fix an orthogonal set $\{e_{ij} \mid i \in \alpha_j\}
  \subseteq \masa$ of equivalent projections that sum to $z_j$ for
  some infinite cardinal $\alpha_j$.
  Partition the infinite set $\{ e_{ij} \mid i \in \alpha_j \}$ into
  two subfamilies of the same cardinality, and let $f_{kj}$ be the sum
  of the $k$th subfamily for $k=1,2$. Then $z_j = f_{1j} + f_{2j}$
  and $f_{1j} \sim f_{2j} \sim z_j$ for all $j$. Set $e = \sum_j
  f_{1j}$ so that $1-e = \sum_j f_{2j}$. Then $e \sim 1-e \sim
  \sum z_j = 1$ as desired.
\end{proof}

\begin{proposition}\label{prop:infinitetypeone}
  Let $\alg$ be an AW*-algebra of type $\typeone_\infty$. For any maximal
  abelian subalgebra $\masa$ of $\alg$, there exists a projection $e \in \masa$ such that
  $e \sim 1 \sim 1 - e$.
\end{proposition}
\begin{proof}
  (We freely use 
  Lemma~\ref{lem:KR6.9.16}  throughout this proof.)
  Let $g \in \masa$ be the supremum of the finite projections in $\masa$. 
  By Lemma~\ref{lem:supoffinites} the nonzero projections in $(1-g) \masa$
  are properly infinite (in $\alg$ and hence) in $(1-g)\alg(1-g)$, so $(1-g)
  = e_1 + e_2$ for orthogonal projections $e_i \in (1-g)\masa$
  with $e_1 \sim e_2 \sim 1 - g$ by Proposition~\ref{prop:infinitemasa}.
  By Lemma~\ref{lem:infiniteincorners} there exists a central projection $z \in \alg$
  such that $zg$ is finite and $(1-z)g$ is properly infinite or zero.
  Then $(1-z)g$ is a supremum of finite projections, so
  Lemma~\ref{lem:KR6.9.24} applied to the maximal abelian subalgebra
  $(1-z)g \masa$ of $(1-z)g \alg (1-z)g$ shows that $(1-z)g = f_1 + f_2$ for
  orthogonal projections $f_i \in (1-z)g \masa$ with $f_1 \sim f_2 \sim (1-z)g$.
  In the sum of orthogonal projections
  \begin{align*}
    1 &= zg + (1-g) + (1-z)g \\
    &= zg + e_1 + f_1 + e_2 + f_2,
  \end{align*}
  set $e = zg + e_1 + f_1 \in \masa$. 
  We will prove below that $zg + (1-g) \sim 1-g$, from which it will
  follow that $(1-g) + (1-z)g \sim z g + (1-g) + (1-z)g = 1$ and thus
  $1- e \sim e \sim 1$.
  
  It remains to show that $zg + (1-g) \sim 1-g$.
  We claim that $z \leq \cover{1-g}$. To see this, note that the central
  projection $y = z (1 - \cover{1-g})$ satisfies $y \leq z$ and $y (1-g)
  = 0$. Hence $y = yg \leq zg$ is finite. Because $\alg$ is
  properly infinite, we conclude $z (1 - \cover{1-g}) = y = 0$,
  or $z \leq \cover{1-g}$. This gives the middle equality in
  \[
    \cover{zg} \leq z = \cover{z(1-g)} \leq \cover{1-g}.
  \]
  If $1-g = 0$ then $zg = 0$ and the claim is verified.
  If $1-g > 0$ then it is properly infinite, and Lemma~\ref{lem:KR6.9.4}(c)
  implies that $zg + (1-g) \sim 1 - g$.
\end{proof}

Finally, we combine the results for types $\typeone_\infty$, $\typetwo_\infty$,
and $\typethree$ to show that a maximal abelian subalgebra of any properly
infinite algebra contains an infinite set of ``diagonal matrix units''.

\begin{lemma}\label{lem:propinftwo}
  If $\masa$ is a maximal abelian subalgebra of a properly infinite
  AW*-algebra $\alg$, then there exists a projection $e \in \masa$
  such that $e \sim 1 \sim 1-e$.
\end{lemma}

\begin{proof}
  By~\cite[Theorem~15.3]{berberian}, $1 = z_1 + z_2$ is a sum of
  orthogonal central projections such that $z_1 \alg$ has type~$\typeone$
  and $z_2 \alg$ has no central summands of type~$\typeone$. Each $z_i \masa$
  is a maximal abelian subalgebra of $z_i \alg$. 
  By Propositions~\ref{prop:infinitenotypeone} and~\ref{prop:infinitetypeone}
  there are projections $f_i \in z_i \masa$ such that $f_i \sim z_i \sim z_i - f_i$.
  Then $e = f_1 + f_2 \in \masa$ satisfies $e \sim 1 \sim 1-e$.
\end{proof}

\begin{theorem}\label{thm:infinitymatrixunits}
  Let $\alg$ be a properly infinite AW*-algebra, let $\masa$ be a maximal abelian
  subalgebra of $\alg$, and let $1 \leq n \leq \aleph_0$ be a cardinal.
  Then there is a set $\{e_i\}$ of $n$ orthogonal projections in $\masa$
  such that $1 = \sum e_i$ and every $e_i \sim 1$. 
\end{theorem}
\begin{proof}
  By Lemma~\ref{lem:downwardclosed}, it suffices to consider the case
  $n = \aleph_0$.
  For all positive integers $k$, we inductively decompose $1 = e_1 + \cdots
  + e_k + f_k$ as a sum of orthogonal projections in $\masa$ where all
  $e_i \sim f_k \sim 1$. For $k = 1$ use Lemma~\ref{lem:propinftwo},
  and for the inductive step suppose we have $e_1, \dots, e_k, f_k \in \masa$
  as above. 
  Lemma~\ref{lem:propinftwo} applied to the maximal abelian subalgebra
  $f_k \masa$ of the properly infinite algebra $f_k \alg f_k$ provides
  a projection $f \in f_k \masa$ such that $f \sim f_k \sim f_k - f$.
  Thus $e_{k+1} = f$ and $f_{k+1} = f_k - f$ satisfy $e_{k+1} \sim f_{k+1}
  \sim f_k \sim 1$ and $1 = e_1 + \cdots + e_{k+1} + f_{k+1}$ as desired.
  So $\{e_i\}_{i = 1}^\infty$ is an orthogonal set of $\aleph_0$ projections
  in $\masa$ that are equivalent to~1. In case $e = \sum e_i \in \masa$ is
  not equal to $1$, we may replace $e_1$ with $e'_1 = e_1 + (1-e) \in \masa$;
  since $1 \sim e_1 \leq e'_1 \leq 1$, Schr\"{o}der--Bernstein implies
  that $e'_1 \sim 1$.
\end{proof}

Theorem~\ref{thm:infinitymatrixunits} naturally suggests the question of
how large the cardinality $n$ of a set of diagonal matrix units in
$\masa$ can become. In other words: given a maximal 
abelian subalgebra $\masa$ of a properly infinite AW*-algebra $\alg$,
for what (infinite) cardinals $n$ does there exist an orthogonal set
$\{e_i\} \subseteq \Proj(\masa)$ of cardinality $n$ such that $\sum e_i
= 1$ and each $e_i \sim 1$? 
It would be interesting to know to what extent the answer depends
upon the particular subalgebra $\masa$.

\section{Simultaneous diagonalization}\label{sec:diagonalization}

We are now ready to prove simultaneous diagonalization over arbitrary
AW*-algebras. 
Recall that if $\alg$ is an AW*-algebra, then so is $\M_n(\alg)$~\cite{berberian:matrices}.

\begin{lemma}\label{lem:KR6.9.31}
  Let $\alg$ be an AW*-algebra of type $\typeone_m$ for a positive
  integer $m$. If $1 = e_1 + \dots + e_n$
  for some equivalent projections $e_i \in \alg$, then $m$ is divisible by $n$.
\end{lemma}
\begin{proof}
  We proceed by induction on $m$. The case $m = 1$ is evident.
  Let $f_1, \dots, f_m$ be orthogonal equivalent abelian projections with sum~1.
  Since $\alg$ is finite, so is each $e_i$; notice also that each $\cover{e_i} = 1$.
  So for each $1 \leq i \leq n$ there is a projection $g_i$ with $f_1
  \sim g_i \leq e_i$~\cite[Corollary 18.1]{berberian}.   
  Now the set $\{g_1, \ldots, g_n\}$ of orthogonal equivalent abelian projections
  has cardinality at most $m$~\cite[Proposition~2]{berberian}.
  The projections $e_i - g_i$ remain equivalent~\cite[Exercise~17.3]{berberian},
  and because $\sum g_i \sim (f_{m-n+1} + \cdots +f_m)$, similarly,
  the projection $h = \sum_{i=1}^n (e_i - g_i) = 1 - \sum g_i$ satisfies
  \[
    h \sim 1 - (f_{m-n+1} + \cdots +f_m) = f_1 + \cdots + f_{m-n}.  
  \]
  On the one hand, $h\alg h$ contains the $n$ orthogonal equivalent
  projections $e_i - g_i$. But on the other hand, it has type $\typeone_{m-n}$
  by the equation above. The inductive hypothesis implies that $m-n$ is divisible
  by $n$, whence $m$ is divisible by $n$.
\end{proof}

\begin{proposition}\label{prop:KR6.9.34}
  Let $\alg$ be an arbitrary AW*-algebra, and $n$ a positive integer. If $\masa$
  is a maximal abelian subalgebra of $\M_n(\alg )$, then it contains $n$ orthogonal
  projections with sum $1$ equivalent in $\M_n(\alg )$.
\end{proposition}
\begin{proof}
  From~\cite[Theorem~15.3]{berberian}
  and~\cite[Theorem~18.3]{berberian}, there are central projections
  $z_1,z_2,\ldots,z_{\text{c}}, z_\infty$ with sum $1$ such that:
  $z_m\M_n(\alg )$ is of type $\typeone_m$ for every $m \in \{1,2,\ldots\}$;
  $z_{\text{c}} \M_n(\alg)$ is of type~$\typetwo_1$; and $z_\infty \M_n(\alg)$
  is properly infinite.
  By Lemma~\ref{lem:KR6.9.31}, $z_m=0$ when $m$ is finite and not
  divisible by $n$. So $z_m > 0$ with $m$ finite implies that $m=kn$
  for some positive integer $k$, and $z_m\masa$ contains $n$ equivalent
  projections $e_{1m},\ldots,e_{nm}$ with sum $z_m$ by
  Proposition~\ref{prop:KR6.9.23}(c).
  Now, $z_{\text{c}} \masa$ contains $n$ equivalent projections
  $e_{1c},\ldots,e_{nc}$ with sum $z_{\text{c}}$ from
  Proposition~\ref{prop:KR6.9.27}, and
  $z_\infty\masa$ contains $n$ equivalent projections $e_{1\infty},\ldots,
  e_{n\infty}$ with sum $z_\infty$ from Theorem~\ref{thm:infinitymatrixunits}.
  Set $e_j=e_c + e_{j \infty} + \sum_{m=1}^\infty e_{jm}$ for
  each $j \in \{1,\ldots,n\}$, where $e_{jm}$ is defined to be $0$ if $m < \infty$
  does not divide $n$. Then $\{e_1\ldots,e_n\}$ is a set of $n$ equivalent
  projections in $\masa$ with sum $1$.
\end{proof}

\begin{lemma}\label{lem:KR6.9.14}
  Let $\alg$ be an AW*-algebra and $\{e_1,\ldots,e_n\}$,
  $\{f_1,\ldots,f_n\}$ be two finite sets of projections in $\alg$,
  both summing to 1, with $e_1\sim\cdots\sim e_n$, $f_1\sim\cdots\sim
  f_n$. Then $e_j \sim f_j$ for each $j=1,\ldots,n$. 
\end{lemma}
\begin{proof}
  Lemma~\ref{lem:infiniteincorners}(b) gives a central projection $z \in \alg$ such
  that $ze_1$ is properly infinite or $z=0$, and $(1-z)e_1$ is
  finite. Then $(1-z)e_2\ldots,(1-z)e_n$ are also
  finite~\cite[Proposition~15.3]{berberian}. By~\cite[Theorem~17.3]{berberian},
  $\sum_{j=1}^n (1-z)e_j = 1-z$ is finite. Hence
  $(1-z)f_1,\ldots,(1-z)f_n$ are finite. If $(1-z)e_1$ is not
  equivalent to $(1-z)f_1$, there is a central projection $y$ in $\alg$
  such that either $y(1-z)e_1 \prec y(1-z)f_1$ or $y(1-z)f_1 \prec
  y(1-z)e_1$. In the former case,
  \[
    y(1-z)e_j \sim g_j < y(1-z)f_j, \qquad \text{ for }j=1,\ldots,n.
  \]
  Hence $y(1-z) \sim \sum g_j < y(1-z)$, contradicting the finiteness
  of $y(1-z)$. Thus $(1-z)e_1 \sim (1-z)f_1$.

  Suppose $z>0$. Then $ze_1$ is properly infinite and by Lemma~\ref{lem:KR6.9.4}(b),
  \[
    ze_1 \sim ze_1 + ze_2 \sim \cdots \sim \sum_{j=1}^n ze_j=z
  \]
  since $ze_1 \sim \cdots \sim ze_n$. Now $zf_1$ is properly infinite,
  for if $z_0\leq p$ is a nonzero central projection with $z_0zf_1$ is
  finite, then $\sum_{j=1}^n z_0f_j$ is finite. But $z_0ze_1$ and,
  hence, $z_0$ are infinite since $ze_1$ is properly infinite,
  contradicting finiteness of $z_0$. Since $zf_1$ is properly
  infinite, as before, $zf_1 \sim z$. Thus $ze_1 \sim zf_1$. It
  follows that $e_1 \sim f_1$.
\end{proof}

The previous lemma is known to hold for Rickart
C*-algebras by a less elementary proof~\cite[Theorem~2.7]{ara:leftright} (see also~\cite[Proposition~13.2]{berberian}). We arrive at our main theorem.

\begin{theorem}\label{thm:KR6.9.35}
  Let $\alg$ be an AW*-algebra, and let $\alg[X] \subseteq \M_n(\alg)$
  be a commuting set of normal elements. There is a unitary $u \in \M_n(\alg )$
  such that $uau^{-1}$ is diagonal for each $a \in \alg[X]$.
\end{theorem}
\begin{proof}
  Let $\masa$ be a maximal abelian subalgebra of $\M_n(\alg )$
  containing $\alg[X]$. By Proposition~\ref{prop:KR6.9.34}, $\masa$ contains $n$
  orthogonal equivalent projections $f_1,\ldots,f_n$ with sum $1$. Let
  $e_{jk}$ be the element of $\M_n(\alg )$ with $1$ at the $(j,k)$-entry
  and $0$ elsewhere. Then $\{e_{11},\ldots,e_{nn}\}$ is an orthogonal
  family of equivalent projections in $\M_n(\alg )$ with sum $1$. By
  Lemma~\ref{lem:KR6.9.14}, $e_{jj} \sim f_j$ for $j=1,\ldots,n$.

  Say $v_j^* v_j = f_j$ and $v_j v_j^* = e_{jj}$ with $v_j \in
  \M_n(\alg )$. Then $u=\sum_{j=1}^n v_j$ is a unitary element of
  $\M_n(\alg )$ and $uf_ju^{-1} = e_{jj}$. Since $f_j$ commutes with every
  element in $\masa \supseteq \alg[X]$, we see that $e_{jj}$ commutes with
  $uau^{-1}$ for all $a \in \alg[X]$. Thus $uau^{-1}$ is diagonal for all $a
  \in \alg[X]$.
\end{proof}

It seems natural to ask whether the converse of Theorem~\ref{thm:KR6.9.35} holds, in the
following sense. To use a term defined in Section~\ref{sec:introduction}, the above theorem
says that an AW*-algebra is simultaneously $n$-diagonalizable for all positive integers $n$.
Conversely, if a C*-algebra $\alg$ is simultaneously $n$-diagonalizable for all $n$, does it
follow that $\alg$ is an AW*-algebra? The question is especially tantalizing because Grove
and Pedersen have answered this question affirmatively in the case where $\alg$ is
commutative, even under the weaker assumption that $\alg$ is simultaneously
$n$-diagonalizable for one fixed $n \geq 2$ (see Theorem 2.1 and the ``Notes added
in proof'' of~\cite{grovepedersen:diagonalizing}).
But it seems a subtle problem to decide the issue in a fully noncommutative setting.

\section{Passing to matrix rings is functorial}\label{sec:functor}

Denote by $\cat{AWstar}$ the category of AW*-algebras and
$*$-homomorphisms between them that preserve suprema of arbitrary
sets of projections. As a consequence of our main result,
Theorem~\ref{thm:KR6.9.35}, we can now show that passing to matrix
rings is an endofunctor on this category.

\begin{lemma}\label{lem:morphisms}
  A $*$-homomorphism $f \colon \alg \to \alg[B]$ between AW*-algebras preserves
  suprema of arbitrary families of projections if and only if it
  preserves suprema of orthogonal families of projections.
\end{lemma}
\begin{proof}
  One direction is trivial. For the other, suppose that $f$ preserves
  suprema of orthogonal sets of projections, and let $\{p_i\}$ be an arbitrary
  family of projections in $\alg$. Then $\ker(f)=z\alg$ for
  some central projection $z\in \alg$~\cite[Exercise~23.8]{berberian}.  
  Hence $\alg  = z\alg  \oplus (1-z)\alg$. Since $p_i=zp_i+(1-z)p_i$
  for each $i$, this yields $\bigvee p_i = \bigvee zp_i + \bigvee (1-z)p_i$,
  and so $f(\bigvee p_i) = f(\bigvee zp_i) + f(\bigvee (1-z)p_i)
  = f(\bigvee (1-z)p_i)$.
  Therefore we may pass to $(1-z)A$ and assume that $\ker(f)=\{0\}$.
  The proof in the case where $f$ is injective is~\cite[Exercise~4.27]{berberian}. 
\end{proof}

\begin{theorem}\label{thm:functor}
  There is a functor $\cat{AWstar} \to \cat{AWstar}$ extending the assignment
  $\alg \mapsto \M_n(\alg)$ on objects.
\end{theorem}
\begin{proof}
  It is well known that if $\alg$ is an AW*-algebra, then $\M_n(\alg)$
  is, too~\cite{berberian:matrices}, and that if $f \colon \alg \to
  \alg[B]$ is a $*$-homomorphism, then $\M_n(f) \colon \M_n(\alg) \to
  \M_n(\alg[B])$ is, too. The point is to show that $\M_n(f)$ preserves 
  suprema of projections. By Lemma~\ref{lem:morphisms} it suffices to
  show that $\M_n(f)$ preserves suprema of orthogonal families of
  projections. Let $\{p_i\}$ be an orthogonal family of projections in
  $\M_n(\alg )$. Then $\{p_i\}$ is an abelian self-adjoint subset of
  $\M_n(\alg )$. Theorem~\ref{thm:KR6.9.35} provides a unitary $u \in
  \M_n(\alg )$ making each $up_iu^{-1}$ diagonal. Now
  \begin{align*}
           \M_n(f)\left(\sum p_i\right) 
    & = \M_n(f)(u)^{-1} \cdot \M_n(f)\left(\sum up_iu^{-1}\right) \cdot \M_n(f)(u)  \\
    & = \M_n(f)(u)^{-1} \cdot \sum \M_n(f)(up_iu^{-1}) \cdot \M_n(f)(u) \\
    & = \sum \M_n(f)(p_i).
  \end{align*}
  The first and last equalities hold because $\M_n(f)$ is a
  $*$-homomorphism, and the middle equality holds because $upu^{-1}$ is
  diagonal, $f$ preserves suprema of projections, and the supremum of a set of
  diagonal projections is computed entrywise.
\end{proof}

\begin{remark}\label{rem:baer}
  The previous theorem holds unabated if we replace $*$-homomorphisms
  by $*$-ring homomorphisms. In fact, due to the algebraic nature of our
  methods, the results in Sections~\ref{sec:preliminaries}--\ref{sec:comparison:infinite}
  seem to hold for Baer $*$-rings with generalized comparability (GC) that satisfy
  the parallellogram law (P), in Berberian's terms~\cite{berberian}.
  For the results of Section~\ref{sec:diagonalization} and the proof of 
  Theorem~\ref{thm:functor} to carry through, one must further restrict to such
  $*$-rings $A$ for which $\M_n(A)$ is again such a $*$-ring (for instance, properly
  infinite Baer $*$-rings $A$, where $\M_n(A) \cong A$ for all $n$). 
  Lemma~\ref{lem:morphisms} additionally requires the ``weak existence of projections''
  (WEP) axiom of~\cite{berberian}.
\end{remark}

\appendix
\section{Achieving the dimension}
\label{app:supachieved}

This appendix discusses two special cases in which the supremum in the
definition of the dimension of properly infinite projections in an
AW*-algebra $\alg$, Definition~\ref{def:tau}, is achieved with certainty. To
be precise, fix a properly infinite projection $e \in A$, and define
\begin{align*}
  \Delta(e) 
  & = \{ \card{I} \mid \{e_i\}_{i \in I} \in \Gamma(e) \} 
  = \{ \text{cardinals } \beta \mid \beta < \d(e) \}, \\
  \delta(e) 
  & = \sup \Delta(e).
\end{align*}
The question is whether $\delta(e) \in \Delta(e)$, or equivalently, whether
$\d(e) = \alpha^+$ for some $\alpha \in \Delta(e)$.
The first special case in which we have a positive answer concerns properties of the cardinal
$\delta(e)$ itself.
Recall that a cardinal is \emph{weakly inaccessible} if it is an uncountable regular limit cardinal.

\begin{proposition}
 If the cardinal $\delta(e)$ is not weakly inaccessible, then $\delta(e) \in \Delta(e)$.
\end{proposition}
\begin{proof}
 Notice from Lemma~\ref{lem:alphasquared} that $\aleph_0 \in \Delta(e)$ necessarily.
 If $\delta(e)$ is either $\aleph_0$ or a successor cardinal, then from the
 definition $\delta(e) = \sup \Delta(e)$ it is clear that $\delta(e) \in \Delta(e)$.
 So we may assume that $\delta(e)$ is an uncountable limit cardinal that is
 not regular: it is strictly larger than the least cardinality of a cofinal set of
 cardinals below it.
 Let $\{\alpha_i \mid i \in \beta\}$ be a cofinal set of cardinals below
 $\delta(e)$, where $\beta<\delta$. Because $\beta<\delta(e)$, we can
 write $e=\sum_{i \in \beta} e_i$ with $e_i \sim e$. Next, we can
 also write $e_i =  \sum_{j \in \alpha_i} e_{ij}$ for each $i \in
 \beta$ with $e_{ij} \sim e_i \sim e$. Then $\{e_{ij}\} \in
 \Gamma(e)$ has cardinality $\sup \{ \alpha_i \mid i \in \beta\}$, which
 equals $\delta(e)$ by cofinality. 
\end{proof}

The answer is also positive when $\alg$ is a von Neumann algebra.
Recall that a projection is countably decomposable when any orthogonal
family of nonzero subprojections is countable;
$\alg$ is countably decomposable when $1_{\alg}$ is.

\begin{lemma}\label{lem:sumofcountables}
  Every projection $p$ in a von Neumann algebra $\alg$ can be written as
  $p=\sum p_i$ for some orthogonal family $\{p_i\}$ of countably
  decomposable projections. Hence every central projection $z$ can be
  written as $z=\sum z_i$ for an orthogonal family $\{z_i\}$ of
  central projections making each $z_i Z(\alg)$ countably decomposable.
\end{lemma}
\begin{proof}
  Let $\alg$ act faithfully on a Hilbert space $H$.  Then $p =
  \sum p_i$ for an orthogonal set $\{p_i\}$ of projections 
  cyclic in $\alg$ for the action on $H$~
  \cite[Proposition~5.5.9]{kadisonringrose}.  
  But every cyclic projection in $\alg$ is countably
  decomposable~
  \cite[Proposition~5.5.15]{kadisonringrose}.
\end{proof}

\begin{lemma}\label{lem:cutcoversum}
  Let $\alg$ be a von Neumann algebra.
  \begin{enumerate}[(a)]
  \item If $\{e_i \mid i \in \alpha\} \subseteq \alg$ is an orthogonal set
    of equivalent countably decomposable nonzero projections
    for some infinite cardinal $\alpha$, then the projection $e = \sum_{i \in \alpha} e_i$ 
    is properly infinite and satisfies $\alpha = \delta(e) \in \Delta(e)$. 
  \item If $e \in \alg$ be a nonzero properly infinite projection, then there
    exists a nonzero central projection 
    $z \leq \cover{e}$ and an infinite set of projections $\{e_i \mid i \in \alpha\}$ as in~(a) such
    that $ze = \sum_{i \in \alpha} e_i$. 
  \end{enumerate}
\end{lemma}
\begin{proof}
  For (a): by Lemma~\ref{lem:alphasquared}, $e$ is
  properly infinite and $\alpha \in \Delta(e)$. On the other
  hand, if $\beta \in \Delta(e)$ 
  there is a family $\{f_j\}_{j \in \beta} \in \Gamma(e)$. Because
  $
  \sum_{i \in \alpha} e_i = e = \sum_{j \in \beta} f_j
  $
  and the $e_i$ are countably decomposable, an easy adaptation of the proofs
  of~\cite[Lemma~1]{tomiyama:dimension}
  and~\cite[Lemma~6.3.9]{kadisonringrose} 
  shows that $\beta \leq \alpha$.
  Because $\beta \in \Delta(e)$ was arbitrary, it follows that $\delta(e) \leq \alpha$.
  But $\alpha \in \Delta(e)$ further implies that $\delta(e) = \alpha \in \Delta(e)$.

  For (b): by Lemma~\ref{lem:sumofcountables}, $\cover{e}$ is a sum of central
  projections $\{z_i\}$ making each $z_iZ(\alg)$ countably
  decomposable. Passing to a direct summand, we may assume that 
  $\cover{e} Z(\alg)$ itself is countably decomposable. Then $\cover{e} =
  \cover{p}$ for some countably decomposable projection $p \in
  \alg$~\cite[Propositions~5.5.16 and~5.5.15]{kadisonringrose}. 

  Write $e = \sum_{j=1}^\infty f_j$ with $f_j \sim e$.  Notice that
  $p$ is countably decomposable, $f_j \sim e$ are properly infinite,
  and $\cover{p} = \cover{e} = \cover{f_j}$. It follows
  from~\cite[Theorem~6.3.4]{kadisonringrose} that $p
  \precsim e \sim f_j$. So there exist orthogonal $g_j \leq f_j \leq
  e$ with $g_j \sim p$ for  each $j$. Extend $\{g_j\}_{j=1}^\infty$
  via Zorn's lemma to a maximal orthogonal set of projections 
  $\{g_i \mid i \in \alpha\}$ such that $p \sim g_i \leq e$ for all $i
  \in \alpha$, where $\alpha$ is an infinite cardinal.
  Assume for contradiction that $e - \sum g_i$ is
  properly infinite with central cover $\cover{e}$. Then $p \precsim e
  - \sum g_i$
  by~\cite[Theorem~6.3.4]{kadisonringrose}.
  This allows us to enlarge the set $\{g_i\}$, contradicting maximality. 
  Therefore the situation reduces the following two cases.

  \emph{Case 1: $e - \sum g_i$ is not properly infinite.} Then there is a
  nonzero central projection $z \leq \cover{e - \sum g_i} \leq
  \cover{e}$ making $z(e - \sum g_i) > 0$ finite. Because $\cover{g_i}
  = \cover{e}$ for each $i$, it follows that $\cover{\sum g_i} =
  \cover{e} \geq \cover{e - \sum g_i}$. Note that $\sum zg_i$ is
  properly infinite by Lemma~\ref{lem:alphasquared}(a). Furthermore,
  $z(e - \sum g_i)$ and $\sum zg_i$ have central cover $z$. It follows from 
  Lemma~\ref{lem:KR6.9.4}(c) that
  \[
    ze = z\left(e - \sum g_i\right) + \sum zg_i \sim \sum zg_i.
  \]
  Thus $ze = \sum e_i$ for equivalent countably decomposable $e_i \sim
  zg_i$. 

  \emph{Case 2: $\cover{e - \sum g_i}$ is strictly below $\cover{e}$.}
  Define $z = \cover{e} -  \cover{e - \sum g_i}$. Then $0 < z
  \leq \cover{e}$, and $z(e - \sum g_i) = 0$. Thus
  $ze = \sum zg_i$, where the $e_i=zg_i$ are pairwise equivalent and
  countably decomposable. 
  %
\end{proof}

\begin{proposition}\label{prop:achievingtauwstar}
  If $e$ is properly infinite projection in a von Neumann algebra $\alg$,
  then $\delta(e) \in \Delta(e)$.
\end{proposition}
\begin{proof}
  Applying Zorn's lemma to Lemma~\ref{lem:cutcoversum}(b) gives a maximal
  family $\{z_i\}$ of orthogonal nonzero central projections such that
  $z_i \leq \cover{e}$, and $z_ie = \sum_j e_{ij}$ for some
  infinite orthogonal set $\{e_{ij} \mid j \in \alpha_i\}$ of equivalent
  countably decomposable projections. 
  If $\sum z_i < \cover{e}$, then ($\cover{e} - \sum z_i)e$ is properly
  infinite, so the projection given by Lemma~\ref{lem:cutcoversum}(b)
  would violate maximality; therefore $\sum z_i = \cover{e}$.
  Lemma~\ref{lem:cutcoversum}(a) also implies
  that $z_i e$ is properly infinite and $\delta(z_i e) =
  \alpha_i$. Then $\delta(e) = \min\{\alpha_i\} \in \Delta(e)$ by Lemma~\ref{lem:tau}(c).
\end{proof}

\bibliographystyle{amsplain}
\bibliography{diagonal}

\end{document}